%% file: JMCentre.tex
\documentclass{amsart}
\input{preamble}

\title[The center conjecture for cyclotomic Hecke algebras]{The center conjecture for cyclotomic Hecke algebras and cohomology of quiver varieties}

 \author{Jun Hu}
 \address{Key Laboratory of Algebraic Lie Theory and Analysis of Ministry of Education\\
 School of Mathematics and Statistics\\
 Beijing Institute of Technology\\
 Beijing, 100081, P.R. China}
 \email{junhu404@bit.edu.cn}

\author{Andrew Mathas}
\address{School of Mathematics and Statistics\\
University of Sydney\\
NSW 2006, Australia}
\email{andrew.mathas@sydney.edu.au}

\author{Lei Shi}
\address{Max-Planck-Institut f\"ur Mathematik\\
Vivatsgasse 7, 53111 Bonn\\
Germany}

\email{leishi202406@163.com}
\author{Ben Webster}
\address{Department of Pure Mathematics, University of Waterloo \& Perimeter Institute for Theoretical Physics, Waterloo, ON, Canada}
\email{ben.webster@uwaterloo.ca}

\date{\today}

\begin{document}

\begin{abstract}
For a quiver without edge loops, we construct an isomorphism from the cohomology of a Nakajima quiver variety with coefficients in an arbitrary field to the Jucys--Murphy center of the corresponding cyclotomic KLR algebra over that field. In affine type $A^{(1)}_{e-1}$, combining this with a dimension count shows that this Jucys--Murphy center coincides with the full center in characteristic $0$, when the characteristic does not divide $e$, and when the characteristic equals $e$.  This follows from a geometric theorem showing that the sum of these cohomology rings is cocyclic as a module over an integral form of a loop algebra.

As a consequence, we prove that the center of a cyclotomic Hecke algebra of type $G(r,1,n)$ with $q\neq 1$ over any field is generated by symmetric polynomials in its Jucys--Murphy elements, and the same holds over any commutative domain in which $q-1$ is invertible.
We also prove a generalization to the integral case of the conjecture of
Krylov, Paegelow and Shlykov identifying the equivariant $K$-theory of the
core of a Gieseker variety with the cocenter of the generic cyclotomic
Hecke algebra, including the natural extension of this conjecture to
positive characteristic.
\end{abstract}

\maketitle

\tableofcontents

\section{Introduction}
\label{sec:intro}

The $(\Lambda-\alpha)$-weight space of an irreducible highest weight
representation of a Kac--Moody algebra~$\mathfrak{g}$ carries two
natural commutative algebras, one geometric and one algebraic. On the geometric side is the cohomology ring $H^*(\fM_{\Lambda,\alpha})$ of a Nakajima quiver variety. On the algebraic side is the center $Z(\Rn[\alpha])$ of the cyclotomic KLR algebra $\Rn[\alpha]$, which categorifies this weight space. Here
$\mathfrak{g}$ is the symmetric Kac--Moody algebra of a finite graph
$\Gamma=(I,E)$, without edge loops and possibly with multiple edges,~$\Lambda$ is a dominant weight, and $\alpha\in Q^+$ lies in the positive root cone.

Since the cyclotomic KLR algebra $\Rn[\alpha]$ is an algebraic analogue of the quiver variety $\fM_{\Lambda,\alpha}$, it is natural to expect a connection between these two algebras.  Furthermore, the isomorphism of Brundan and Kleshchev \cite{BKKL} between cyclotomic KLR algebras and blocks of cyclotomic Hecke algebras intertwines this with the question of understanding the center of a cyclotomic Hecke algebra, a question of longstanding interest in representation theory.

The isomorphism $\phi\colon H^*(\fM_{\Lambda,\alpha};\C)\to Z(\Rn[\alpha](\C))$ was proven in finite type A by Brundan \cite{Brundan:degenCentre} and in general finite type by Shan, Varagnolo and Vasserot \cite{shanCenterQuiver2017} and the fourth-named author \cite{Webcenter}.  The latter two proofs used techniques in the representation theory of current algebras which cannot be applied in affine type A.

The approach of \cite{Webcenter} is greatly simplified by the work of McGerty and Nevins \cite{mcgertyKirwanSurjectivity2018} on Kirwan surjectivity.  The quiver varieties $\fM_{\Lambda,\alpha}$ carry tautological bundles $V_i$ for each vertex $i\in I$, whose Chern classes
 $c_{i;j}\in H^{2j}(\fM_{\Lambda,\alpha};\Z)$ give natural elements of the cohomology ring.  McGerty and Nevins showed that the Kirwan map from the polynomial ring in these classes to $H^*(\fM_{\Lambda,\alpha};\Z)$ is surjective.  Localization to torus fixed points determines its kernel in principle, but at the moment, no closed-form presentation of this kernel seems to be known.

Assume that $\alpha\in Q_n^+$, that is, $\alpha$ is a sum of $n$ simple roots.  For each $i\in I, 1\leq j\leq n$, define $e_{i;j}:=\sum_{\nu\in I^\alpha}e_j(y_a\mid\nu_a=i)e(\nu)\in Z(\Rn[\alpha])$, where $I^\alpha$ is the set of sequences $\nu\in I^n$ with $\sum_a\al_{\nu_a}=\alpha$; this is the $j$th elementary symmetric polynomial in the dots on the strands with label $i$ in the KLR algebra $\Rn[\alpha]$. We call the subalgebra generated by the $e_{i;j}, i\in I, 1\leq j\leq n$, the \emph{Jucys--Murphy center}, or \emph{JM center}.  The main result of this paper is the following:

\begin{itheorem}\label{thm:A}
  Let $\Gamma$ be a finite graph without edge loops and let $\kk$ be any field of characteristic $p\geq 0$.
  \begin{enumerate}
    \item There is an algebra homomorphism
  \[\phi_{\kk}\colon H^*(\fM_{\Lambda,\alpha};\kk)\to Z(\Rn[\alpha](\kk))\]
  with $\phi_{\kk}(c_{i;j})=e_{i;j}$ for each $i\in I, 1\leq j\leq n$.  This map is an isomorphism onto the Jucys--Murphy center of $\Rn[\alpha](\kk)$.  In particular, it is always injective.
  \item The map $\phi_\kk$ is an isomorphism if and only if the JM center is the full center.  These equivalent conditions hold if and only if
  \[\dim H^*(\fM_{\Lambda,\alpha};\kk)\geq \dim Z(\Rn[\alpha](\kk)).
  \]
  \item If every connected component of $\Gamma$ is an $e$-cycle, for the same $e$, and either $p=0$, $p\nmid e$, or $p=e$, then $\phi_{\kk}$ is an isomorphism.
  \end{enumerate}
\end{itheorem}

The conditions on $p$ in (3) might look a little strange. These are
exactly the characteristics in which the required dimension of the
center is known.  If $p=0$ or $p\nmid e$, the cyclotomic KLR algebra is
a block of a non-degenerate cyclotomic Hecke algebra, to which we apply
the dimension count of \cref{thm:HS}. If $p=e$, we use the analogous count for a degenerate cyclotomic Hecke algebra due to Brundan \cite{Brundan:degenCentre}.  This gives the upper bound required in part (2).

Let $\Hn$ be the cyclotomic Hecke algebra (Ariki--Koike algebra) of type
$G(r,1,n)$ over a commutative ring $\kk$ with Hecke parameter $q\neq 1$
and cyclotomic parameters $\bQ=(Q_1,\dots,Q_r)\in(\kk^\times)^r$; see
\cref{def:Hecke}. It is a well-known result of Jucys~\cite[(8)]{Jucys}
and Murphy~\cite[Theorem~1.9]{Murphy:Nakayama} that the center of the
group algebra of the symmetric group is the set of symmetric polynomials
in the Jucys--Murphy elements of this algebra.
Brundan~\cite[Theorem~1]{Brundan:degenCentre} proved the corresponding
result for the degenerate Ariki--Koike algebras. Hence, it is natural to
make the following conjecture, which was part of the folklore long
before it first appeared in the literature in~\cite{McGerty12}.

\begin{conjecture}[Center conjecture~{\cite[Conjecture~3.1]{McGerty12}}]
\label{conj:Hecke-centre}
Suppose that $q-1\in\kk^\times$. Then the center of $\Hn$ is the set of
symmetric polynomials in the Jucys--Murphy elements $L_1,\dots,L_n$.
\end{conjecture}

Equivalently, \cref{conj:Hecke-centre} says that the center of the
affine Hecke algebra of type~$A$ surjects onto the center of the
Ariki--Koike algebras. In the semisimple case,
\cref{conj:Hecke-centre} was proved by Ariki and Koike
\cite[Theorem~3.20]{ArikiKoike94}. McGerty~\cite[Theorem~2.5]{McGerty12}
established the conjecture when $n=2$. When $r=1$, the conjecture was
proved by Francis and Graham~\cite{FG}. Importantly,
Ariki~\cite[\S2, Example]{Ariki:can} gives an example to show that the
assumption $q\ne1$ in \cref{conj:Hecke-centre} is necessary.

Together with the isomorphism of Brundan and Kleshchev \cite{BKKL}, \cref{thm:A}(3) gives the \emph{center conjecture} for the Ariki--Koike algebras:

\begin{itheorem}\label{thm:B}
  Suppose that $\kk$ is a field and $q\ne1$. Then the center of $\Hn$ is the algebra of symmetric polynomials in its Jucys--Murphy elements $L_1,\dots,L_n$.
\end{itheorem}

Since \cref{thm:B} holds in every characteristic, we can apply \cref{thm:B} to prove a generalization of a conjecture of Krylov, Paegelow and Shlykov \cite{KrylovPaegelowShlykov26} that relates the equivariant $K$-theory of the Gieseker variety $\mathcal{G}(n,r)$, the Nakajima quiver variety of the Jordan quiver, and of its core $\Lrn$ to the center and cocenter of the generic cyclotomic Hecke algebra. Write $\SAc=\Z[q^{\pm1},Q_1^{\pm1},\dots,Q_r^{\pm1},(q-1)^{-1}]$; see \cref{SS:KTheory} for the unexplained notation below.

\begin{samepage}
\begin{itheorem}\label{thm:C}\hfill
  \begin{enumerate}
    \item Let $\kk$ be a commutative domain with $q,Q_1,\dots,Q_r,q-1\in
    \kk^\times$.  Then \[Z(\Hn(\kk))=Z(\Hn(\kk))^{\JM}.\] In particular, this holds over $\SAc$.
    \item There are compatible isomorphisms:
    \[\SAc\otimes\KT(\mathcal{G}(n,r))\cong Z(\Hn(\SAc))\qquad\text{and}\qquad\SAc\otimes\KT(\Lrn)\cong\Tr(\Hn(\SAc)).\]
    \item Let $e\geq 2$ and $\kk$ be a field of characteristic coprime to $e$ and $\zeta\in\kk$ be a primitive $e$th root of unity. Let $\kappa_1,\dots,\kappa_r\in\Z$, and let $\Lambda=\sum_m\Lambda_{\bar\kappa_m}$ for the quiver of type $A^{(1)}_{e-1}$.  For every $\alpha\in Q^+_n$, there is an isomorphism of $\kk$-algebras $K(\NLa;\kk)\cong Z(\HL[\alpha](\kk))$ between the $K$-theory of the quiver variety and the center of the corresponding block of the cyclotomic Hecke algebra over $\kk$ with $q=\zeta$ and $Q_m=\zeta^{\kappa_m}$.
  \end{enumerate}
\end{itheorem}
\end{samepage}

For $\kk=\C$, part (3) is \cite[Th.~5.4]{KrylovPaegelowShlykov26}, which was proved assuming \cite[Conj.~4.6]{KrylovPaegelowShlykov26}; part (3) removes this assumption.  The matching of the components of the fixed locus with blocks of the Hecke algebra follows Paegelow \cite[Th.~2.15, Prop.~2.17 and Prop.~5.3]{Paegelow26}.

The key new ingredient is injectivity over an arbitrary field of the map from cohomology to $Z(\Rn[\alpha])$; this means we can rule out the existence of non-JM classes in the center if we can bound the dimension from above.  This follows for the same fundamental reason as the characteristic-zero injectivity in \cite{Webcenter}: $\phi$ is not just a map of rings, but of representations of a larger algebra.  In \cref{thm:cyclic}, we extend a theorem of Nakajima \cite[Prop.~12.3.2, 13.3.1]{nakajimaQuiverVarieties2001a} to guarantee that $H^*(\fM_{\Lambda})=\bigoplus_\alpha H^*(\fM_{\Lambda,\alpha})$ is cogenerated by $H^*(\fM_{\Lambda,0})\cong \kk$.
This is dual to the more geometrically natural statement that
the homology of the core $\fL_{\Lambda,\alpha}=\pi^{-1}(0)\subset \fM_{\Lambda,\alpha}$ is generated by $H_*(\fM_{\Lambda,0})\cong \kk$.

Thus, some explanation is needed about how our results here are different from \cite{Webcenter}:
\begin{itemize}
  \item Since \cite{Webcenter} was written when Kirwan surjectivity was not known, it used the interpretation of $H^*(\fM_{\Lambda,\alpha})$ as the Hochschild cohomology of a category of sheaves over a quantization of the structure sheaf of $\fM_{\Lambda,\alpha}$.  This constructed $\phi$ in characteristic $0$, even on classes outside the image of the Kirwan map.

Given that we now know Kirwan surjectivity holds, we can more straightforwardly define an integral map directly on the Chern-class generators; we use the characteristic-zero map of \cite{Webcenter} to verify that their relations are preserved.
  \item For reduction to positive characteristic, we need a more careful analysis of the geometry of quiver varieties and Hecke correspondences, resulting in the stronger integral statement proved in \cref{thm:cyclic}: the Chow groups of the cores are generated over $\Z$ by decorated divided-power correspondences.  We also identify the integral form of the loop algebra generated by these divided powers and compare its geometric and KLR actions.
  \item A few simplifications made in that paper exclude the case of $e=2$; our arguments here cover all $e>1$.
\end{itemize}

The proof of \cref{thm:A}(1)--(2) has three steps.  First, in \cref{sec:quiver}, we define decorated divided-power Hecke correspondences on the Chow groups of the cores of quiver varieties and prove integral cyclicity (\cref{thm:cyclic}).  Next, in \cref{sec:comparison}, we transpose these to raising and lowering operators on integral cohomology, identify the integral form of the loop algebra they generate, and compare its action with the corresponding operators on the KLR side.  Finally, \cref{sec:injective} uses Kirwan surjectivity to construct the integral comparison map and cyclicity to prove its injectivity after arbitrary base change.  Part (3) and \cref{thm:B} are proved in \cref{sec:affineA}. \cref{thm:C} is proved in \cref{SS:KTheory}.

\section{Quiver varieties and Chow groups}
\label{sec:quiver}

\subsection{Notation}\label{sec:notation}
We follow the conventions of \cite{nakajimaQuiverVarieties2001a}.  Fix $\Gamma=(I,E)$ as above, together with dimension vectors $\Bv=(v_i)_{i\in I},\Bw=(w_i)_{i\in I}\in \Z_{\geq 0}^I$.  For each $k\in I$, let $\Be_k\in \Z_{\geq 0}^I$ denote the $k$th coordinate vector.  Corresponding to these dimension vectors, let $\fM(\Bv,\Bw)$ and $\fM_0(\Bv,\Bw)$ be the quiver varieties as defined in \cite[(2.2.2),(2.2.1)]{nakajimaQuiverVarieties2001a}, $\pi\colon\fM(\Bv,\Bw)\to \fM_0(\Bv,\Bw)$ be the affinization map. We have the dominant weight $\Lambda=\sum_i w_i\Lambda_i$, the element $\alpha=\sum_i v_i\al_i\in Q^+$, and the weight $\mu=\Lambda-\alpha$, and we write
\[\fM(\Bv,\Bw)=\fM_{\Lambda,\alpha},\quad \fL(\Bv,\Bw)=\pi^{-1}(0)=\fL_{\Lambda,\alpha},\quad
\fM_{\Lambda}=\bigsqcup_{\alpha\in Q^+}\fM_{\Lambda,\alpha}, \quad \fL_{\Lambda}=\bigsqcup_{\alpha\in Q^+}\fL_{\Lambda,\alpha}.
\]
For $n\geq 0$, we write \[\fM_{\Lambda,n}=\bigsqcup_{\alpha\in Q^+_n}\fM_{\Lambda,\alpha}.\] Here $Q^+_n$ is the set of $\alpha=\sum_iv_i\al_i\in Q^+$ of height $\sum_iv_i=n$.  This means, in particular, that \[H^*(\fM_{\Lambda})=\bigoplus_\alpha H^*(\fM_{\Lambda,\alpha})=\bigoplus_n H^*(\fM_{\Lambda,n}),\] and similarly for $\fL_{\Lambda}$.

Throughout, we will consider varieties $X$ over $\C$ and let $A_*(X)$ denote the Chow group of $X$ with integer coefficients, and use $H_*(-)$ (resp., $H_*^{\rm{ord}}(-)$) to denote the Borel--Moore homology (resp., the ordinary singular homology) with integer coefficients unless another ring is indicated.  For compact spaces, such as $\fL(\Bv,\Bw)$, Borel--Moore homology $H_*(-)$ agrees with ordinary singular homology $H_*^{\rm{ord}}(-)$.

\begin{lemma}\label{lem:pairing}\hfill
 \begin{enumerate}
\item $\fM(\Bv,\Bw)$ is a smooth quasi-projective variety, and $\fL(\Bv,\Bw)$ is a projective subvariety onto which $\fM(\Bv,\Bw)$ retracts \cite[Prop.~2.3.8]{nakajimaQuiverVarieties2001a}.
\item The varieties $\fM(\Bv,\Bw)$ and $\fL(\Bv,\Bw)$ have property (S) of \cite[\S 7.1]{nakajimaQuiverVarieties2001a}: their odd homology vanishes, their even homology is a free abelian group, and the cycle map $\cl\colon A_*\to H_{2*}$ is an isomorphism \cite[Th.~7.3.5]{nakajimaQuiverVarieties2001a}.
\item For any field $\kk$, we have isomorphisms
\begin{equation}\label{eq:UCT}
H^*(\fM(\Bv,\Bw);\kk)=H^*(\fM(\Bv,\Bw);\Z)\otimes \kk, \qquad H_*(\fL(\Bv,\Bw);\Z)\cong H_*^{\rm{ord}}(\fM(\Bv,\Bw);\Z)
\end{equation}
and the induced pairing
  \begin{equation}\label{eq:pairing}
H^k(\fM(\Bv,\Bw);\kk)\otimes H_k(\fL(\Bv,\Bw);\kk)\to \kk,\qquad m\otimes c\mapsto \langle m|_{\fL(\Bv,\Bw)},c\rangle
\end{equation}
  is perfect.
\end{enumerate}
\end{lemma}

\begin{proof}
  Parts (1) and (2) are cited.  For (3), since the inclusion $\fL=\fL(\Bv,\Bw)\hookrightarrow \fM(\Bv,\Bw)$ is a deformation retract and $\fL(\Bv,\Bw)$ is compact, we have $H_*(\fL(\Bv,\Bw);\Z)\cong H_*^{\rm{ord}}(\fL(\Bv,\Bw);\Z)\cong H_*^{\rm{ord}}(\fM(\Bv,\Bw);\Z)$.  \cref{eq:UCT} follows from  (2) by the universal coefficient theorem.  \cref{eq:pairing} follows from the second isomorphism in \eqref{eq:UCT}.
\end{proof}

\subsection{Hecke correspondences}\label{sec:hecke-correspondences}

Following the conventions of Nakajima, we represent a point in $\fM(\Bv,\Bw)$ by a collection of matrices $x=[B,i,j]$, considered up to base change by the group $G_{\Bv}$.   We now fix a vertex $k\in I$ and an integer $n\geq 1$ with $v_k\geq n$.

We will make heavy use of the Hecke correspondence
\[\fP^{(n)}_k=\fP^{(n)}_k(\Bv,\Bw)\subset \fMo\times \fMt\]
of \cite[\S 5.3]{nakajimaQuiverVarieties2001a}.  If we regard $\fM(\Bv,\Bw)$ as the moduli space of stable framed representations of the preprojective algebra, then $\fP^{(n)}_k$ is the moduli space of pairs $V^1\subset V^2$ of a representation $V^2$ of dimension $\Bv$ and a subrepresentation $V^1$ of dimension $\Bv-n\Be_k$ containing the image of $i$ \cite[(5.3.1)]{nakajimaQuiverVarieties2001a}, with the projections $p_1$ and $p_2$ recording $V^1$ and $V^2$, respectively.  This is a closed subvariety such that the projections $p_1,p_2$ to both factors are proper; in fact, the fibers are always Grassmannians.

An important role in understanding these fibers is played by the complex of \cite[(2.9.1)]{nakajimaQuiverVarieties2001a}
\[V_k\xrightarrow{\ \sigma_k\ }\bigoplus_{\mathrm{in}(h)=k}V_{\mathrm{out}(h)}\oplus W_k\xrightarrow{\ \tau_k\ }V_k,\qquad \sigma_k=\bigoplus_{\mathrm{in}(h)=k}B_{\bar h}\oplus j_k,\qquad \tau_k=\sum_{\mathrm{in}(h)=k}\varepsilon(h)B_h+i_k,\]
where $\bar h$ is $h$ with the opposite orientation and $\varepsilon$ is the sign function of \cite[\S 2.1]{nakajimaQuiverVarieties2001a}.  Thus $\tau_k$ is the sum of all the maps into $V_k$ and $\sigma_k$ is the sum of all the maps out of it; the composition $\tau_k\sigma_k$ is the $k$th component of the moment map, so it vanishes, and stability implies that $\sigma_k$ is injective.
The fiber of $p_2$ over a point $V^2$ is the Grassmannian of subspaces of $V^2_k$ of dimension $v_k-n$ containing the image of $\tau_k$.  On the other hand, the fiber of $p_1$ over a point $V^1$ of $\fMo$ is the Grassmannian of $n$-dimensional subspaces of $\ker\tau_k/\operatorname{im}\sigma_k$, where now $\sigma_k$ and $\tau_k$ are the maps in the complex above for the representation $V^1$.  Indeed, since $V^1$ and $V^2$ differ only at the vertex $k$, the middle terms of their complexes coincide, and so do the kernels of $\tau_k$.  The map $\sigma_k$ of $V^2$ identifies $V^2_k$ with a subspace of $\ker\tau_k$ containing $\operatorname{im}(\sigma_k)$, of dimension $\dim\operatorname{im}(\sigma_k)+n$, and choosing $V^2\supset V^1$ amounts to choosing this subspace.

Let $\mathcal{S}$ be the tautological vector bundle of rank $n$ defined by $V_k^2/V_k^1$ on $\fP^{(n)}_k$, given by the quotient representation in the description discussed above.
Let $\SymR_n=\Z[x_1,\dots,x_n]^{\Sym_n}\cong\Z[c_1,\dots,c_n]$.  This ring acts on the Chow groups of $\fP^{(n)}_k$ (and of any scheme over it) by evaluating $c_j$ at $c_j(\mathcal{S})$ and capping with the resulting class.

\begin{definition}\label{def:F}
  For $f\in \SymR_n$, the \emph{decorated divided-power lowering operator} is the map
  \[F^{(n)}_k[f]\colon A_*(\fL(\Bv-n\Be_k,\Bw))\to A_*(\fL(\Bv,\Bw)),\qquad F^{(n)}_k[f](c)=q_*\,\delta^!\big(c\times (f\cap [\fP^{(n)}_k])\big).\]
  Here $\delta$ is the diagonal map of the smooth variety $\fMo$, a regular embedding, and $\delta^!$ is the refined Gysin map \cite[\S 6.2]{fultonIntersectionTheory1998} for the fiber square
  \[\begin{tikzcd}
  \fL(\Bv-n\Be_k,\Bw)\times_{\fMo}\fP^{(n)}_k \arrow[r]\arrow[d,"\mathrm{pr}_1"] & \fL(\Bv-n\Be_k,\Bw)\times \fP^{(n)}_k \arrow[d,"\mathrm{incl}\times p_1"]\\
  \fMo \arrow[r,"\delta"] & \fMo\times \fMo
  \end{tikzcd}\]
  and $q\colon \fL(\Bv-n\Be_k,\Bw)\times_{\fMo}\fP^{(n)}_k\to \fL(\Bv,\Bw)$ is the proper map induced by $p_2$; here $\mathrm{pr}_1$ and $p_1$ are the projections onto the first factor.

  Similarly, we have a decorated divided-power raising operator \[E^{(n)}_k[f]\colon A_*(\fL(\Bv,\Bw))\to A_*(\fL(\Bv-n\Be_k,\Bw))\] on Chow groups, defined by the same formula with the roles of $p_1$ and $p_2$ reversed.
\end{definition}
The same formula in Borel--Moore homology defines operators $F^{(n)}_k[f], E^{(n)}_k[f]$ on $H_*(\fL;\Z)$, namely convolution with the class $f\cap[\fP^{(n)}_k]$ in the sense of \cite[\S 2.7]{CG97}:
\begin{equation}\label{eq:pull-push0}
\begin{aligned}
    F^{(n)}_k[f]&\colon H_*(\fL(\Bv-n\Be_k,\Bw))\to H_*(\fL(\Bv,\Bw)), & F^{(n)}_k[f](c)&=(p_2)_*\bigl(p_1^*(c)\cap (f\cap  [\fP^{(n)}_k])\bigr),\\
    E^{(n)}_k[f]&\colon H_*(\fL(\Bv,\Bw))\to H_*(\fL(\Bv-n\Be_k,\Bw)), & E^{(n)}_k[f](c)&=(p_1)_*\bigl(p_2^*(c)\cap (f\cap  [\fP^{(n)}_k])\bigr).
\end{aligned}
\end{equation}
Since the cycle map commutes with exterior products, proper pushforward, refined Gysin maps and Chern classes \cite[Ch.~19]{fultonIntersectionTheory1998}, the cycle map $\cl\colon A_*(\fL)\to H_*(\fL)$ intertwines the two versions.

On cohomology, we define the corresponding operators by pull-push: for $f\in\SymR_n$, let
\begin{equation}\label{eq:pull-push}
\begin{aligned}
    F^{(n)}_k[f]&\colon H^*(\fMo)\to H^*(\fMt), & F^{(n)}_k[f](m)&=(p_2)_*\big(f(\mathcal{S})\cup p_1^*m\big),\\
    E^{(n)}_k[f]&\colon H^*(\fMt)\to H^*(\fMo), & E^{(n)}_k[f](m)&=(p_1)_*\big(f(\mathcal{S})\cup p_2^*m\big),
\end{aligned}
\end{equation}
where $(p_1)_*$ and $(p_2)_*$ are the Gysin maps of the proper maps $p_1,p_2$ between smooth varieties.  Under Poincar\'e duality on the smooth varieties $\fP^{(n)}_k$, $\fMo$ and $\fMt$, these are intertwined with the operators
of \cref{def:F} on the Borel--Moore homology of $\fM$.  We write $E_{k,r}=E^{(1)}_k[x^r]$ and $F_{k,r}=F^{(1)}_k[x^r]$.
\begin{lemma}\label{lem:adjoint}
  Under the perfect pairing of \cref{lem:pairing}, the operator $F^{(n)}_k[f]$ on $H_*(\fL)$ is adjoint to the operator $E^{(n)}_k[f]$ on $H^*(\fM)$ of \eqref{eq:pull-push}, and similarly with $E$ and $F$ reversed.
\end{lemma}
\begin{proof}
  Let $c\in H_*(\fL(\Bv-n\Be_k,\Bw))$ and $m\in H^*(\fMt)$, and put $\gamma=f\cap[\fP^{(n)}_k]$.  The projection formula and the compatibility of refined pullback with cap products and proper pushforward in Borel--Moore homology \cite[\S\S 2.6--2.7]{CG97} give
  \begin{align*}
    \big\langle m,F^{(n)}_k[f](c)\big\rangle
    &=\deg\big(m\cap q_*\delta^!(c\times\gamma)\big)
    =\deg\big(q^*m\cap\delta^!(c\times\gamma)\big)\\
    &=\deg\,\delta^!\big(c\times(p_2^*m\cap\gamma)\big)
    =\big\langle (p_1)_*\big(f(\mathcal{S})\cup p_2^*m\big),c\big\rangle\\
    &=\big\langle E^{(n)}_k[f](m),c\big\rangle.
  \end{align*}
  The statement with $E$ and $F$ exchanged is the same computation with the two factors exchanged.
\end{proof}

\subsection{The stratification by \texorpdfstring{$\varepsilon_k$}{epsilon k}}\label{sec:epsilon-stratification}
As in the previous subsection, we keep $k\in I$ fixed.  For each point $x=[B,i,j]\in \fM(\Bv,\Bw) $, we consider the statistic $\varepsilon_k(x)=\operatorname{codim}_{V_k}\operatorname{Im}\tau_k$.  Following \cite[\S 5.4]{nakajimaQuiverVarieties2001a}, let
\[\fMf{k}{n}=\{x\in \fM(\Bv,\Bw)\mid \varepsilon_k(x)=n\},\qquad \fM_{k;\leq n}=\bigcup_{m\leq n}\fM_{k;m},\qquad \fM_{k;\geq n}=\bigcup_{m\geq n}\fM_{k;m},\]
and similarly $\fL_{k;n}=\fL\cap \fM_{k;n}$ and so on.  The subset $\fM_{k;\leq n}$ is open, so $\fM_{k;\geq n}$ is closed, and $\fM_{k;n}$ is closed in $\fM_{k;\leq n}$.
In particular, for an irreducible closed subvariety $Y\subset \fM(\Bv,\Bw)$, the function $\varepsilon_k$ takes its minimum value on a nonempty open subset of $Y$; we denote this generic value by $\varepsilon_k(Y)$.  Thus $Y\subset \fM_{k;\geq \varepsilon_k(Y)}$.

This filtration is a geometric avatar of decomposing a representation
according to the action of the  $\mathfrak{sl}_2$ generated by the root spaces for $\pm \al_k$, with the open subset $\fM_{k; 0}$ playing the role of the highest weight vectors.

  Write $\fMo^\circ=\fM_{k;0}(\Bv-n\Be_k,\Bw)$.   On $\fM_{k;n}(\Bv,\Bw)$, there is a unique subrepresentation of dimension $\Bv-n\Be_k$, so the map $p_2$ is injective on the preimage of this subset.
 Let $\mathcal{Q}$ be the vector bundle on $\fMo^\circ$ defined by the middle cohomology $\mathcal{Q}_k(\Bv-n\Be_k,\Bw)=\ker\tau_k/\operatorname{Im}\sigma_k$
   of the complex \cite[(2.9.1)]{nakajimaQuiverVarieties2001a}, a vector bundle of rank $m=\langle\al_k^\vee,\mu\rangle+2n$.
\begin{lemma}[{\cite[Prop.~5.4.3]{nakajimaQuiverVarieties2001a}}]\label{lem:grassmann-bundle}
	We have an isomorphism of $\fMf{k}{n}$ to the Grassmannian bundle $G(n,\mathcal{Q})\to \fMo^\circ$ of $n$-planes in $\mathcal{Q}$.  Under this isomorphism, $\fP^{(n)}_k\cap\big(\fMo\times \fM_{k;\leq n}(\Bv,\Bw)\big)$ is the graph of the bundle map.

\end{lemma}

The same result holds with $\fM$ replaced by the core $\fL$, since $\fL_{k;n}(\Bv,\Bw)$ is the restriction of the Grassmann bundle to $\fL_{k;0}(\Bv-n\Be_k,\Bw)$, by the remark before \cite[Prop.~5.4.3]{nakajimaQuiverVarieties2001a} and \cite[Lem.~2.5.3]{nakajimaQuiverVarieties2001a}.

As observed in the proof of \cite[Prop.~12.3.2]{nakajimaQuiverVarieties2001a}, the lemma below follows from \cite[Lem.~2.9.4(1)]{nakajimaQuiverVarieties2001a}.
\begin{lemma}\label{lem:epsilon}
  If $\Bv\neq 0$ and $x\in \fL(\Bv,\Bw)$, then $\varepsilon_k(x)>0$ for some $k\in I$.
\end{lemma}

Whenever we take a Grassmannian bundle of a vector bundle, the effect on Chow groups is easily understood as a relative version of the Schubert cell description of the Chow group of the Grassmannian itself.

 Let $\mathcal{Q}$ be a vector bundle of rank $m$ on a scheme $B$, let $p\colon G(n,\mathcal{Q})\to B$ be the Grassmann bundle of $n$-planes, and let $\mathcal{S}$ be its tautological subbundle.

 Let $\SymR_{n,m-n}$ be the polynomials symmetric under the Young subgroup $\Sym_n\times\Sym_{m-n}$.  For a partition $\nu$ let $\Delta_\nu(c)=\det(c_{\nu_a+b-a})_{a,b}$, a polynomial with integer coefficients in the Chern classes of a vector bundle.  As a $\SymR_m$-module, $\SymR_{n,m-n}$ is freely generated by the $\Delta_\nu(c)$, with $c_j$ the $j$th elementary symmetric polynomial in $x_1,\dots,x_n$, where $\nu$ runs over partitions with at most $m-n$ parts, each part at most $n$.
 This is relevant for us since it is the equivariant Chow ring $A_*^{GL_m}(G(n,\C^m))\cong \SymR_{n,m-n}$.
\begin{lemma}\label{lem:grassmann}
  The group $A_*(G(n,\mathcal{Q}))$ is spanned by the classes $f(\mathcal{S})\cap p^*a$ for $f\in\SymR_n$ and $a\in A_*(B)$.
\end{lemma}
\begin{proof}
  By Fulton's basis theorem \cite[Prop.~14.6.5]{fultonIntersectionTheory1998}, every class in $A_*(G(n,\mathcal{Q}))$ is a sum of classes $\Delta_\lambda\cap p^*a_\lambda$ with $a_\lambda\in A_*(B)$, where $\Delta_\lambda$ is a Schur determinant in the Chern classes of $p^*\mathcal{Q}/\mathcal{S}-p^*\mathcal{Q}$.  Since $c(p^*\mathcal{Q}/\mathcal{S}-p^*\mathcal{Q})=c(\mathcal{S})^{-1}$, these Chern classes are polynomials in those of $\mathcal{S}$.  See also \cite[Prop.~4.5.1]{AndersonFulton24} for the corresponding presentation of the cohomology of a Grassmann bundle.
\end{proof}
This lemma builds in our convolution mechanism:
\begin{lemma}\label{lem:convolution-restriction}
	The map $x\otimes f\mapsto F_k^{(n)}[f](x)$ defines a map $A_*(\fL(\Bv-n\Be_k,\Bw))\otimes \SymR_n \to A_*(\fL(\Bv,\Bw))$. Restriction to the open subsets $\fL_{k;0}(\Bv-n\Be_k,\Bw)$ and $\fL_{k;\leq n}(\Bv,\Bw)$ gives a commutative diagram
	\[\begin{tikzcd}
	A_*(\fL(\Bv-n\Be_k,\Bw))\otimes \SymR_n \arrow[rr,"F^{(n)}_k"]\arrow[d,"\mathrm{res}\otimes 1"'] && A_*(\fL(\Bv,\Bw))\arrow[d,"\mathrm{res}"]\\
	A_*(\fL_{k;0}(\Bv-n\Be_k,\Bw))\otimes \SymR_n \arrow[r,"\psi"] & A_*(\fL_{k;n}(\Bv,\Bw)) \arrow[r,"\iota_*"] & A_*(\fL_{k;\leq n}(\Bv,\Bw))
	\end{tikzcd}\]
	where $\psi(c\otimes f)=f(\mathcal{S})\cap p^*c$ for $p\colon \fL_{k;n}(\Bv,\Bw)\to \fL_{k;0}(\Bv-n\Be_k,\Bw)$ the Grassmannian bundle above and $\mathcal{S}$ its tautological subbundle, and $\iota\colon \fL_{k;n}(\Bv,\Bw)\hookrightarrow \fL_{k;\leq n}(\Bv,\Bw)$ is the closed inclusion.  In particular, the maps $\iota_*$ and $\mathrm{res} \circ F^{(n)}_k $ have the same image.
\end{lemma}

\begin{proof}
  Restriction to an open subset is flat pullback, which commutes with proper pushforward, refined Gysin maps, and capping by Chern classes \cite[Prop.~1.7, Th.~6.2(b)]{fultonIntersectionTheory1998}.  Thus we may compute the restriction of $F^{(n)}_k[f](c)$ to $\fL_{k;\leq n}(\Bv,\Bw)$ using the restriction of $\fP^{(n)}_k$ to $\fMo\times \fM_{k;\leq n}(\Bv,\Bw)$.  By \cref{lem:grassmann-bundle}, this restriction is the graph of the smooth morphism
  \[g\colon \fMf{k}{n}\xrightarrow{p}\fMo^\circ\hookrightarrow \fMo,\]
  and $\mathcal{S}$ is the tautological subbundle of the Grassmann bundle $p$.  Since the image of $g$ lies in $\fMo^\circ$, we may replace $\fMo$ by $\fMo^\circ$ and $c$ by its restriction $c^\circ$ to $\fL_{k;0}(\Bv-n\Be_k,\Bw)$.

  Refined intersection with the graph of $g$ sends $c^\circ\times[\Gamma_g]$ to the refined pullback of $c^\circ$ along $g$ \cite[\S 8.1]{fultonIntersectionTheory1998}.  Since $g$ is smooth, this is the flat pullback $g^*c^\circ$.  Its restriction to $g^{-1}(\fL_{k;0})=\fL_{k;n}$ is $p^*c^\circ$.  Capping with $f(\mathcal{S})$ commutes with these operations, and the map $q$ of \cref{def:F} restricts to $\iota$.  This proves commutativity.
  Finally, $\psi$ is surjective by \cref{lem:grassmann-bundle,lem:grassmann}.  The map $\mathrm{res}\otimes 1$ is surjective because restriction to an open subset is surjective on Chow groups \cite[Prop.~1.8]{fultonIntersectionTheory1998}.  Since $\operatorname{res}\circ F^{(n)}_k=\iota_*\circ\psi\circ(\mathrm{res}\otimes 1)$, the maps $\iota_*$ and $\operatorname{res}\circ F^{(n)}_k$ have the same image.
\end{proof}

\subsection{Cyclicity of the Chow groups of the core}
\label{sec:cyclic}

The following is an integral, Chow-theoretic version of \cite[Prop.~12.3.2]{nakajimaQuiverVarieties2001a}.
\begin{theorem}\label{thm:cyclic}
  For every $\Bv$, the group $A_*(\fL(\Bv,\Bw))$ is spanned over $\Z$ by the classes
  \[F^{(n_1)}_{k_1}[f_1]\,F^{(n_2)}_{k_2}[f_2]\cdots F^{(n_s)}_{k_s}[f_s]\,[\fL(0,\Bw)],\]
  where $\sum_a n_a\Be_{k_a}=\Bv$ and $f_a\in\SymR_{n_a}$.
\end{theorem}
\begin{proof}
  Let $\mathcal{A}(\Bv)\subset A_*(\fL(\Bv,\Bw))$ denote the span of these classes; note that $F^{(n)}_k[f](\mathcal{A}(\Bv-n\Be_k))\subset \mathcal{A}(\Bv)$ for any $f\in\SymR_n$.  We show $\mathcal{A}(\Bv)=A_*(\fL(\Bv,\Bw))$ by induction on $|\Bv|=\sum_k v_k$.
  If $\Bv=0$, then $\fL(0,\Bw)$ is a point and there is nothing to prove.

  Suppose $\Bv\neq 0$ and that the claim holds for all smaller dimension vectors.  The group $A_*(\fL(\Bv,\Bw))$ is generated by the classes $[Y]$ of irreducible closed subvarieties $Y\subset \fL(\Bv,\Bw)$.  For such $Y$, let $N(Y)=\max_k\varepsilon_k(Y)$, where $\varepsilon_k(Y)$ is the generic value of $\varepsilon_k$ on $Y$.  By \cref{lem:epsilon}, we have $N(Y)\geq 1$; and, of course, $N(Y)\leq \max_k v_k$ so the possible values of $N(Y)$ are bounded above.

 For each fixed $\Bv$, we prove that $[Y]\in \mathcal{A}(\Bv)$ by descending induction on $N(Y)$.  So fix an integer $n$, and assume that  $[Y']\in \mathcal{A}(\Bv)$ for every irreducible closed $Y'\subset \fL(\Bv,\Bw)$ with $N(Y')>n$.

Now consider $Y$ such that $N(Y)=n$, so for some $k$, we have $\varepsilon_k(Y)=n$.   Let $U=\fL_{k;\leq n}(\Bv,\Bw)$ and $Z=\fL_{k;\geq n+1}(\Bv,\Bw)$, which are complementary open and closed subsets of $\fL(\Bv,\Bw)$.  Note that by assumption, $Y^{\circ}=Y\cap U\neq \emptyset$ is an open dense subset of $Y$, and a closed subset of $ \fL_{k;n}(\Bv,\Bw) $.

By \cref{lem:grassmann-bundle}, $\fL_{k;n}(\Bv,\Bw)$ is a Grassmann bundle over $\fL_{k;0}(\Bv-n\Be_k,\Bw)$ with tautological subbundle $\mathcal{S}$.  By \cref{lem:grassmann},
  \[[Y^{\circ}]=\sum_\ell f_\ell(\mathcal{S})\cap p^*a_\ell,\qquad f_\ell\in\SymR_n,\ a_\ell\in A_*(\fL_{k;0}(\Bv-n\Be_k,\Bw)).\]
    Since $\fL_{k;0}(\Bv-n\Be_k,\Bw)$ is open in $\fL(\Bv-n\Be_k,\Bw)$, restriction $A_*(\fL(\Bv-n\Be_k,\Bw))\to A_*(\fL_{k;0}(\Bv-n\Be_k,\Bw))$ is surjective \cite[Prop.~1.8]{fultonIntersectionTheory1998}, so we can choose $b_\ell\in A_*(\fL(\Bv-n\Be_k,\Bw))$ restricting to $a_\ell$.  By the inductive hypothesis on $|\Bv|$, $b_\ell\in \mathcal{A}(\Bv-n\Be_k)$.  Let
  \[D=\sum_\ell F^{(n)}_k[f_\ell](b_\ell)\in \mathcal{A}(\Bv).\]
  Now, we consider the behavior of this class in the exact sequence (\cite[Prop.~1.8]{fultonIntersectionTheory1998})
  \[A_*(Z)\to A_*(\fL(\Bv,\Bw))\to A_*(U)\to 0.\]
  Applying the commutative diagram of \cref{lem:convolution-restriction} to each term of $D$, and using $b_\ell|_{\fL_{k;0}(\Bv-n\Be_k,\Bw)}=a_\ell$, we get
  \[D|_U=\iota_*\psi\Big(\sum_\ell a_\ell\otimes f_\ell\Big)=\iota_*\Big(\sum_\ell f_\ell(\mathcal{S})\cap p^*a_\ell\Big)=\iota_*[Y^{\circ}]=[Y]|_U,\]
  where the last equality holds since $Y^\circ=Y\cap U$ is closed in $\fL_{k;n}(\Bv,\Bw)$.  So by exactness, $[Y]-D$ is the pushforward of a class on $Z$.  That is, it is an integral combination of classes $[Y']$ of irreducible closed subvarieties $Y'\subset Z$.  Each such $Y'$ has $\varepsilon_k(Y')\geq n+1$, so $[Y']\in \mathcal{A}(\Bv)$ by induction. This completes the proof.
\end{proof}
This shows that as a module over these operators, the homology of the core is cyclically generated.

\begin{corollary}\label{cor:cyclic-homology}
  For any field $\kk$, the space $H_*(\fL_{\Lambda};\kk)=A_*(\fL_{\Lambda})\otimes_\Z \kk$ is spanned by the classes obtained from the point class $[\fL(0,\Bw)]$ by applying the operators $F^{(n)}_k[f]$ for $f\in \SymR_n$.
\end{corollary}

Since the homology of the core is dual to the cohomology of the quiver variety, we obtain a dual statement about the cohomology of the quiver variety.
\begin{lemma}\label{lem:cyclic-cohomology}
  For any field $\kk$, the space $H^*(\fM_{\Lambda};\kk)$ is cogenerated by the point class $[\fM(0,\Bw)]$; that is, any non-zero subspace of $H^*(\fM_{\Lambda};\kk)$ invariant under the operators $E^{(n)}_k[f]$ for $f\in \SymR_n$ contains the point class.
\end{lemma}
\begin{proof}
  For $v\in H^*(\fM_{\Lambda};\kk)$ non-zero, choose $\alpha$ such that $v$ has non-zero projection to $H^*(\fM_{\Lambda,\alpha};\kk)$ and the weight $\mu=\Lambda-\alpha$ is minimal with this property.  By \cref{cor:cyclic-homology} and the perfect pairing of \cref{lem:pairing}, there is a sequence of operators $F^{(n_1)}_{k_1}[f_1],\dots,F^{(n_s)}_{k_s}[f_s]$ such that
  \[\big\langle v,F^{(n_1)}_{k_1}[f_1]\cdots F^{(n_s)}_{k_s}[f_s][\fL(0,\Bw)]\big\rangle\neq 0.\]
  By the minimality of $\mu$, the image $E^{(n_s)}_{k_s}[f_s]\cdots E^{(n_1)}_{k_1}[f_1]v$ lies in $H^*(\fM(0,\Bw);\kk)$: these operators carry the component of $v$ of weight $\nu$ to weight $\nu+\alpha$, which is a weight of $H^*(\fM_{\Lambda};\kk)$ only if $\nu\leq \mu=\Lambda-\alpha$, so every component other than that of weight $\mu$ is killed.
  By \cref{lem:adjoint}, the non-zero pairing
\[\big\langle E^{(n_s)}_{k_s}[f_s]\cdots E^{(n_1)}_{k_1}[f_1]v, [\fL(0,\Bw)]\big\rangle=\big\langle v,F^{(n_1)}_{k_1}[f_1]\cdots F^{(n_s)}_{k_s}[f_s][\fL(0,\Bw)]\big\rangle\]
shows that $E^{(n_s)}_{k_s}[f_s]\cdots E^{(n_1)}_{k_1}[f_1]v$ is a non-zero element of $H^*(\fM(0,\Bw);\kk)$, so that the submodule generated by $v$ contains the point class.
\end{proof}

Let $x_1,\dots,x_n$ be the Chern roots of $\mathcal{S}$, and write an element $f\in \SymR_n$ as $f=\sum_{\mathbf{a}}f_{\mathbf{a}}x_1^{a_1}\cdots x_n^{a_n}$.  For $n=1$, write $F_{k,r}=F^{(1)}_k[x_1^r]$.
A flag variety computation shows the following.
\begin{lemma}\label{lem:geometric-divided}
  For $f\in \SymR_n$, we have
  \[\sum_{\mathbf{a}}f_{\mathbf{a}}\,F_{k,a_1}F_{k,a_2}\cdots F_{k,a_n}=(-1)^{\binom n2}\,n!\,F^{(n)}_k[f]\]
  as operators $H_*(\fL(\Bv-n\Be_k,\Bw);\Z)\to H_*(\fL(\Bv,\Bw);\Z)$.
\end{lemma}
\begin{proof}
  Let $\tilde{\fP}$ be the variety of chains $V^1=U^0\subset U^1\subset\cdots\subset U^n=V^2$ of subrepresentations with $\dim U^j=\Bv-(n-j)\Be_k$, where $(V^1,V^2)$ is a point of $\fP^{(n)}_k$.  Since $V^1$ and $V^2$ differ only at the vertex $k$, a chain of subspaces $V^1_k\subset U^1_k\subset\cdots\subset V^2_k$ is automatically a chain of subrepresentations: every arrow into $k$ has image in $V^1_k$, and every arrow out of $k$ lands in $V^2_j=V^1_j$ for $j\neq k$.  So the forgetful map $\rho\colon\tilde{\fP}\to\fP^{(n)}_k$ is the bundle of complete flags in $\mathcal{S}=V^2_k/V^1_k$.  It is smooth and proper, with fibers of dimension $\binom{n}{2}$, and the line bundles $U^j_k/U^{j-1}_k$ have first Chern classes $x_j$ which are Chern roots of $\rho^*\mathcal{S}$.

  Write $\fM_j=\fM(\Bv-(n-j)\Be_k,\Bw)$ for $0\leq j\leq n$, and $P_j=\fP^{(1)}_k\subset \fM_{j-1}\times\fM_j$.  The variety $\fM$ is a fine moduli space of stable framed representations, and $P_j$ is the moduli space of pairs $(V^{j-1}\subset V^j)$ of stable representations.  So the scheme-theoretic fiber product $P_1\times_{\fM_1}P_2\times_{\fM_2}\cdots\times_{\fM_{n-1}}P_n$ is the moduli space of chains $V^0\subset V^1\subset\cdots\subset V^n$, that is, it is $\tilde{\fP}$.  In particular this fiber product is smooth and reduced.

  Put $P=P_1\times\cdots\times P_n$, a closed subvariety of $\prod_{j=1}^{n}(\fM_{j-1}\times\fM_j)$, and let
  \[\tilde\delta\colon\fM_0\times\fM_1\times\cdots\times\fM_n\to\prod_{j=1}^{n}(\fM_{j-1}\times\fM_j),\qquad(m_0,\dots,m_n)\longmapsto\big((m_0,m_1),(m_1,m_2),\dots,(m_{n-1},m_n)\big).\]
We can identify $\tilde\delta=\mathrm{id}_{\fM_0}\times\delta_1\times\cdots\times\delta_{n-1}\times\mathrm{id}_{\fM_n}$, where $\delta_j\colon\fM_j\to\fM_j\times\fM_j$ is the diagonal.  So $\tilde\delta$ is a regular embedding with normal bundle $\bigoplus_{j=1}^{n-1}T\fM_j$, and $\tilde{\fP}=\tilde\delta^{-1}(P)$ is the fiber of $\tilde\delta$ over $P$.  The composite $F_{k,a_1}\cdots F_{k,a_n}$ applies the refined Gysin map of \cref{def:F} for the diagonal of each of $\fM_0,\dots,\fM_{n-1}$ in turn.  Those for $\fM_1,\dots,\fM_{n-1}$ are the maps $\delta_j^!$, and the one for $\fM_0$ is the one in the convolution below.  By the push-forward property and because refined Gysin maps of a composite of regular embeddings compose \cite[Th.~6.2(a), Th.~6.5]{fultonIntersectionTheory1998}, the composite is convolution, in the sense of \cref{def:F}, with the class
  \[\gamma=\rho'_*\big(x_n^{a_1}\cdots x_1^{a_n}\cap\tilde\delta^![P]\big)\in A_*(\fM_0\times\fM_n),\]
  where $\rho'\colon\tilde{\fP}\to\fM_0\times\fM_n$ is the projection.  The classes $x_j$ are pulled back from $P$, so they can be taken out of $\tilde\delta^!$.  Since $\tilde{\fP}$ and $P$ are smooth, the inclusion $\tilde{\fP}\hookrightarrow P$ is a regular embedding.  The excess intersection formula \cite[Th.~6.3]{fultonIntersectionTheory1998} gives $\tilde\delta^![P]=c_{\mathrm{top}}(E)\cap[\tilde{\fP}]$, where the \emph{excess bundle} $E$ is the cokernel of
  \[\bigoplus_{j=1}^nTP_j\longrightarrow\bigoplus_{j=1}^{n-1}T\fM_j,\qquad(u_j)_j\longmapsto\big(d\pi^{(j)}_2u_j-d\pi^{(j+1)}_1u_{j+1}\big)_j,\]
  restricted to $\tilde{\fP}$ (here $\pi^{(j)}_1,\pi^{(j)}_2$ are the projections of $P_j$).  The rank of $E$ is the difference between the expected dimension of the fiber product and $\dim\tilde{\fP}$.

  We identify $E$ using the symplectic forms.  Each $P_j$ is Lagrangian in $\fM_{j-1}\times\fM_j$ for the form $\omega_{j-1}\boxminus\omega_j$ \cite[Th.~5.7]{nakajimaQuiverVarieties1998}.  The symplectic forms give a perfect pairing between $E$ and the orthogonal complement of the image of the map above.  A tuple $(m_1,\dots,m_{n-1})$, with $m_j\in T\fM_j$, is orthogonal to the image exactly when each $(m_{j-1},m_j)$ is orthogonal to $TP_j$, where we put $m_0=m_n=0$; as $P_j$ is Lagrangian, this says $(m_{j-1},m_j)\in TP_j$.  Thus the orthogonal complement is the tangent space to the fiber of $\tilde{\fP}\to\fM_0\times\fM_n$.  The image of $\tilde{\fP}$ in $\fM_0\times\fM_n$ is $\fP^{(n)}_k$, so this fiber is the fiber of $\rho$, and $E\cong T^*_\rho$.  Hence $c_{\mathrm{top}}(E)=c_{\binom n2}(T_\rho^*)$.  The bundle $T_\rho$ has a filtration with line bundle quotients $\mathrm{Hom}(U^i_k/U^{i-1}_k,U^j_k/U^{j-1}_k)$ for $i<j$, with first Chern classes $x_j-x_i$, so
  \[c_{\mathrm{top}}(E)=\prod_{i<j}(x_i-x_j)=(-1)^{\binom n2}\prod_{i<j}(x_j-x_i).\]

  Integrating out the complete flag, $\rho_*\big(g\cdot\prod_{i<j}(x_j-x_i)\cap[\tilde{\fP}]\big)=\big(\sum_{w\in \Sym_n}w(g)\big)\cap[\fP^{(n)}_k]$ for every polynomial $g$ in the $x_j$ \cite[\S 14.2]{fultonIntersectionTheory1998}: $\rho_*$ is the divided difference $\partial_{w_0}$ and $\prod_{i<j}(x_j-x_i)$ is the Euler class of $T_\rho$.  (For $n=2$ and $g=1$ this says $\rho_*(x_2-x_1)=\chi(\mathbb{P}^1)=2$.)  Since $f$ is symmetric, $\sum_{\mathbf{a}}f_{\mathbf{a}}\sum_{w}w(x_n^{a_1}\cdots x_1^{a_n})=n!\,f$.  Hence $\sum_{\mathbf{a}}f_{\mathbf{a}}F_{k,a_1}\cdots F_{k,a_n}=(-1)^{\binom n2}n!\,F^{(n)}_k[f]$ as desired.
\end{proof}

\section{Comparison with the KLR algebras}\label{sec:comparison}
\subsection{The KLR side}\label{sec:KLR-side}
We will require the following objects from the world of categorified Lie algebras:
\begin{enumerate}
\item The nilHecke algebra $\NH_n$ on $n$ strands over $\Z$, generated by  $y_1,\dots,y_n,\psi_1,\dots,\psi_{n-1}$.  This algebra is isomorphic to the matrix algebra $\operatorname{Mat}_{n!}(\SymR_n)$ over the symmetric polynomial ring $\SymR_n=\Z[y_1,\dots,y_n]^{\Sym_n}$, with the Morita equivalence induced by the idempotent
\[e_n=y_1^{n-1}y_2^{n-2}\cdots y_{n-1}\psi_{w_0}\in \NH_n.\]
\item The KLR algebra $\mathsf{R}_\alpha$ \cite{khovanovCategorificationQuantum2010,Rou2KM} is the $\Z$-algebra generated by idempotents $e(\mathbf{i})$, for sequences $\mathbf{i}=(i_1,\dots,i_n)\in I^n$ with $\sum_{t=1}^{n}\al_{i_t}=\alpha$, together with $y_1,\dots,y_n$ and $\psi_1,\dots,\psi_{n-1}$,
subject to the KLR relations for the polynomials  $Q_{ij}(u,v)=(u-v)^{\#\{j\to i\}}(v-u)^{\#\{i\to j\}}$ for $i\neq j$ in $I$, and $Q_{ii}(u,v)=0$, as in \cite[\S 1.1]{Webcenter}.
Since signs can have a subtle effect on biadjunctions here, we note that, for $i\neq j$ in $I$ and $a_{ij}$ the Cartan matrix entries,
\[Q_{ij}(u,v)=r_{ij}(u-v)^{-a_{ij}}=r_{ji}(v-u)^{-a_{ij}},\quad
-a_{ij}=\#\{j\to i\}+\#\{i\to j\},\quad r_{ij}=(-1)^{\#\{i\to j\}}.\]
In particular $r_{ij}=(-1)^{a_{ij}}r_{ji}$, which is condition (11) of \cite{shanCenterQuiver2017}.

\item The cyclotomic quotient $\Rn[\alpha]=\mathsf{R}_\alpha/\big\langle y_1^{\langle\al_{i_1}^\vee,\Lambda\rangle}e(\mathbf{i})\big\rangle$ for a dominant weight $\Lambda$.  We write \[\mathsf{R}^\Lambda=\bigoplus_{\alpha\in Q^+} \Rn[\alpha]\quad\text{ and }\quad \Rn=\bigoplus_{\alpha\in Q^+_n}\Rn[\alpha],\] so that $Z(\mathsf{R}^\Lambda)=\bigoplus_\alpha Z(\Rn[\alpha])$ and $\Tr(\mathsf{R}^\Lambda)=\bigoplus_\alpha\Tr(\Rn[\alpha])$.  When we regard $\Rn[\alpha]$ as attached to an object of the 2-category $\tU$ below, that object has the weight $\mu=\Lambda-\alpha$. For any field $\kk$, we use $\mathsf{R}_\alpha(\kk)$ and $\mathsf{R}_\alpha^\Lambda(\kk)$ to denote the KLR algebra and the cyclotomic KLR algebra over $\kk$ respectively.
\item The Kac--Moody 2-category $\tU$ of \cite{khovanovCategorificationQuantum2010,Rou2KM}.  We use the cyclic presentation of \cite{beliakovaCyclicityCategorified2016}, since we require a cyclic biadjunction.
\item The divided power 1-morphisms in $\tU$ are defined by $\mathcal{E}_i^{(n)}=(\mathcal{E}_i^n,e_n)$ in the Karoubi envelope \cite{khovanovExtendedGraphical2012}.  The Morita equivalence of the nilHecke algebra referenced above \cite[\S\S 2.5, 4.2]{khovanovExtendedGraphical2012} exactly shows that $\mathcal{E}_i^n\cong (\mathcal{E}_i^{(n)})^{\oplus n!}$ for $n\geq 1$.
\item We use the usual convention of letting $\tU$ be a graded category, and using $\tU^*$ for the 2-category obtained from $\tU$ by forgetting the grading: the same objects and 1-morphisms with degree shifts ignored, and 2-morphisms of all degrees.  Traces (below) are taken in $\tU^*$ and $\Tr(\Rn[\alpha])$ is the trace of the ungraded algebra.   This 2-category has a representation sending $\mu=\Lambda-\alpha$ (as an object) to the category  $\Rn[\alpha]\operatorname{-pmod}$ of ungraded projective modules, with $\mathcal{E}_i$ and $\mathcal{F}_i$ acting by restriction and induction functors.
\end{enumerate}

\begin{lemma}\label{lem:free}
  The groups $\mathsf{R}^\Lambda, Z(\mathsf{R}^\Lambda)$ are free abelian.
\end{lemma}
\begin{proof}
The freeness of $\mathsf{R}^\Lambda$ follows immediately from the fact that the dimension of each $\Rn[\alpha]$ is independent of the field (\cite[Cor. 3.22]{Webmerged}, \cite{HuShiGD24}).  The center $Z(\Rn[\alpha])$ is a subgroup of a free abelian group $\Rn[\alpha]$ of finite rank and hence is free too.
\end{proof}
Note that this does not imply that $\dim Z(\mathsf{R}_\alpha^\Lambda(\kk))$ is
independent of the field $\kk$, since taking the center does not commute with arbitrary base change.

The trace of a $\Z$-linear category $\mathcal{C}$ is
\[\Tr(\mathcal{C})=\Big(\bigoplus_{X}\operatorname{End}_{\mathcal{C}}(X)\Big)\Big/\operatorname{span}\{fg-gf\},\]
where $f\colon X\to Y$ and $g\colon Y\to X$ range over all pairs of morphisms, and we write $[f]$ for the class of an endomorphism $f$.  The trace is unchanged by passing to the additive Karoubi envelope, and for an algebra $A$, we have $\Tr(A):=\Tr(A\operatorname{-proj})\cong A/[A,A]$.  For a linear 2-category, horizontal composition makes the traces of the Hom-categories into a linear category, and $\Tr$ is 2-functorial \cite[\S\S 2.5--2.6]{beliakovaCurrentAlgebras2017}.  So a 2-representation of $\tU$ on a category $\mathcal{C}=\bigoplus_\mu\mathcal{C}_\mu$ makes $\Tr(\mathcal{C})=\bigoplus_\mu\Tr(\mathcal{C}_\mu)$ a module over $\Tr(\tU^*)$, by the formula $[a]\cdot [g]=[a\circ u(g)]$ for $a\in\operatorname{End}(u)$.

Dually, we can consider the centers $Z(\mathcal{C}_{\mu})=\operatorname{End}(\operatorname{Id}_{\mathcal{C}_{\mu}})$ of the $\Z$-linear categories $\mathcal{C}_{\mu}$.  Since we have fixed a duality on $\tU$, which is cyclic (left and right duals coincide), we have an action of $\Tr(\tU^*)$ on $Z(\mathcal{C})=\bigoplus_\mu Z(\mathcal{C}_\mu)$ by ``wrapping'' \cite[\S 9]{beliakovaCurrentAlgebras2017}.

In particular $\Tr(\tU^*)$ acts on $\Tr(\mathsf{R}^\Lambda)$.  The same operators were constructed on the cocenters of cyclotomic quotients by Shan, Varagnolo and Vasserot \cite[\S\S 3.2.5--3.2.6]{shanCenterQuiver2017}.  By \cite[Th.~3.18 and Rem.~3.19]{Webmerged}, we have a
canonical duality between cocenters and centers induced by closing up elements of the cyclotomic quotient.  Under this duality, the induced action of $\Tr(\tU^*)$ on $Z(\mathcal{C})$ is dual to that on $\Tr(\mathcal{C})$ (\cref{lem:adjunction}).

The trace of $\tU^*$ is a version of a current algebra.  For $\Gamma$ of finite type ADE it is the idempotented current algebra $\dot U(\mathfrak{g}[t])$ \cite[Th.~B]{beliakovaCurrentAlgebras2017}.  Outside of finite type, the relations it satisfies are those of the loop algebra of \cref{def:loop} below (\cref{prop:loop-action}).

For $f\in \SymR_n$, let
\[\mathsf{E}^{(n)}_i[f]=[f(y)\,e_n]\in \Tr(\tU^*),\]
the class of the endomorphism $f(y)e_n$ of $\mathcal{E}_i^n1_\mu$, and similarly $\mathsf{E}_{i,r}=\mathsf{E}^{(1)}_i[y^r]$; define $\mathsf{F}^{(n)}_i[f]$ and $\mathsf{F}_{i,r}$ in the same way, using $\mathcal{F}_i^n$.  By the Morita equivalence discussed before, we have an analogue of \cref{lem:geometric-divided} in the trace of $\tU^*$.
\begin{lemma}\label{lem:KLR-divided}
  For $f\in \SymR_n$, we have $\sum_{\mathbf{a}}f_{\mathbf{a}}\,\mathsf{E}_{i,a_1}\cdots\mathsf{E}_{i,a_n}=n!\,\mathsf{E}^{(n)}_i[f]$ in $\Tr(\tU^*)$.
\end{lemma}
\begin{proof}
  Since the product in $\Tr(\tU^*)$ is induced by horizontal composition, the left-hand side is the class $[f(y)]$ of the endomorphism $f(y_1,\dots,y_n)$ of $\mathcal{E}_i^n1_\mu$.  The algebra $\NH_n$ is isomorphic to the matrix algebra $\operatorname{End}_{\SymR_n}(\Z[y_1,\dots,y_n])$ of rank $n!$ over $\SymR_n$, and $e_n$ corresponds to a rank one idempotent \cite{khovanovCategorificationQuantum2010,Rou2KM}.  Thus $1=\sum_{j=1}^{n!}u_ju'_j$ with $u'_ju_j=e_n$ for all $j$.  Since $f$ is central in $\NH_n$,
  \[[f]=\sum_j[fu_ju'_j]=\sum_j[u'_jfu_j]=\sum_j[fu'_ju_j]=n![fe_n].\qedhere\]
\end{proof}

As in \cite[\S 1]{Webcenter} (see also \cite[\S 9]{beliakovaCurrentAlgebras2017}), a cyclic biadjunction on $\tU$ makes $Z(\mathsf{R}^\Lambda)$ a module over $\Tr(\tU^*)$, with $[a]$ acting by the convolution $a\star -$.

\subsection{Comparison of the actions}\label{sec:comparison-actions}
Fix a bilinear form $\beta\colon\Z I\times\Z I\to\Z/2\Z$ with $\beta(\gamma,\delta)+\beta(\delta,\gamma)\equiv\langle\gamma,\delta\rangle\pmod 2$, as in \cite{Webcenter}.  Varagnolo's operators \cite[(4.1)]{varagnoloQuiverVarieties2000}, in the normalization of \cite[(3.1)]{Webcenter}, are
\[x^+_{k,r}=(-1)^{\beta(\mu-\Lambda,\al_k)}(\pi_1)_*\big(z^r\cup\pi_2^*(-)\big),\qquad x^-_{k,r}=(-1)^{\beta(\al_k,\mu-\Lambda)}(\pi_2)_*\big(z^r\cup\pi_1^*(-)\big)\]
on $H^*(\fM_{\Lambda,\alpha};\C)$, where $\pi_1,\pi_2$ are the projections of the Hecke correspondence $\fP^{(1)}_k\subset\fM_{\Lambda,\alpha-\al_k}\times\fM_{\Lambda,\alpha}$ and $z$ is the first Chern class of the quotient $V'_k/V_k$ of the larger by the smaller tautological bundle.  These are our $p_1,p_2$ and $z=c_1(\mathcal{S})$, and the pushforward $\pi_*$ of \cite[\S 3.1]{Webcenter}, defined through Poincar\'e duality on the smooth varieties $\fP^{(1)}_k$ and $\fM$, is the Gysin map in \eqref{eq:pull-push}.  Comparing with \eqref{eq:pull-push}, we thus have

\begin{equation}\label{eq:varagnolo-signs}
  x^+_{k,r}=(-1)^{\beta(\mu-\Lambda,\al_k)}E_{k,r},\qquad x^-_{k,r}=(-1)^{\beta(\al_k,\mu-\Lambda)}F_{k,r}
\end{equation}

on $H^*(\fM_{\Lambda,\alpha};\C)$.
For $\mu$ a weight of the irreducible representation $V(\Lambda)$ of highest weight $\Lambda$, $k\in I$ and $n\geq 1$, define signs
\[\sigma^+(\mu,k,n)=(-1)^{\binom n2}\prod_{t=0}^{n-1}(-1)^{\beta(\mu+t\al_k-\Lambda,\,\al_k)},\qquad \sigma^-(\mu,k,n)=(-1)^{\binom n2}\prod_{t=0}^{n-1}(-1)^{\beta(\al_k,\,\mu-t\al_k-\Lambda)}.\]
Apart from the factor $(-1)^{\binom n2}$ of \cref{lem:geometric-divided}, these are the products of the signs in \eqref{eq:varagnolo-signs} along a string of $n$ raising (respectively lowering) operators starting at weight $\mu$.  Let $\beta'$ be chosen as in the proof of \cite[Th.~3.4]{Webcenter}, so that the $\beta$-geometric and $\beta'$-Cautis--Lauda dualities agree, and let $\tU^*$ act on $Z(\mathsf{R}^\Lambda)$ by convolution with respect to the $\beta'$-Cautis--Lauda duality.
\begin{lemma}\label{lem:intertwine}
  For all $k$, $n$ and $f\in\SymR_n$, the map $\phi$ of \cite{Webcenter} intertwines $\sigma^+(\mu,k,n)E^{(n)}_k[f]$ with $\mathsf{E}^{(n)}_k[f]\star-$, and $\sigma^-(\mu,k,n)F^{(n)}_k[f]$ with $\mathsf{F}^{(n)}_k[f]\star-$, on the summand of weight $\mu$.
\end{lemma}
\begin{proof}
  For $n=1$, \cite[Th.~3.4]{Webcenter} (combining Corollaries~2.8 and 3.3 there) says that $\phi$ intertwines Varagnolo's operators $x^\pm_{k,r}$ with $\mathsf{E}_{k,r}\star-$ and $\mathsf{F}_{k,r}\star-$; by \eqref{eq:varagnolo-signs}, $x^\pm_{k,r}=\sigma^\pm(\mu,k,1)E_{k,r}$, respectively $\sigma^\pm(\mu,k,1)F_{k,r}$.  A product of $n$ such twisted operators starting at weight $\mu$ is $(-1)^{\binom n2}\sigma^\pm(\mu,k,n)$ times the corresponding untwisted product, since $\sigma^\pm(\mu,k,n)$ contains the extra factor $(-1)^{\binom n2}$.  Transposing \cref{lem:geometric-divided} and its analogue with the factors exchanged gives $\sum_{\mathbf{a}}f_{\mathbf{a}}E_{k,a_n}\cdots E_{k,a_1}=(-1)^{\binom n2}n!E^{(n)}_k[f]$, and similarly for $F$, over $\Q$.  So $\sum_{\mathbf{a}}f_{\mathbf{a}}x^\pm_{k,a_n}\cdots x^\pm_{k,a_1}=n!\,\sigma^\pm(\mu,k,n)E^{(n)}_k[f]$, respectively with $F$.  Comparing with \cref{lem:KLR-divided}, using that convolution is an action of $\Tr(\tU^*)$, and dividing by $n!$, which is possible over $\Q$, gives the claim for all $n$.
\end{proof}

\subsection{The loop algebra}\label{sec:loop-algebra}
Let $(a_{ij})_{i,j\in I}$ be the Cartan matrix of $\Gamma$, so that $-a_{ij}$ is the number of edges joining $i\neq j$.
\begin{definition}\label{def:loop}
  Let $L\mathfrak{g}$ be the Lie algebra over $\Q$ generated by elements $x^\pm_{i,r}$ and $\xi_{i,r}$, for $i\in I$ and $r\geq 0$, subject to the following relations for all $i,j\in I$ and $r,s\geq 0$:
  \begin{align}
    [\xi_{i,r},\xi_{j,s}]&=0, \tag{L1}\\
    [\xi_{i,r},x^\pm_{j,s}]&=\pm a_{ij}\,x^\pm_{j,r+s}, \tag{L2}\\
    [x^+_{i,r},x^-_{j,s}]&=\delta_{ij}\,\xi_{i,r+s}, \tag{L3}\\
    [x^\pm_{i,r},x^\pm_{i,s}]&=0, \tag{L4}\\
    \sum_{p=0}^{m}(-1)^p\binom{m}{p}\,[x^\pm_{i,r+m-p},x^\pm_{j,s+p}]&=0 &&\text{for }i\neq j,\ m=-a_{ij}, \tag{L5}\\
    [x^\pm_{i,r_1},[x^\pm_{i,r_2},\dots,[x^\pm_{i,r_m},x^\pm_{j,s}]\cdots]]&=0 &&\text{for }i\neq j,\ m=1-a_{ij},\ r_1,\dots,r_m\geq 0. \tag{L6}
  \end{align}
\end{definition}
In terms of the series $x^\pm_i(z)=\sum_{r\geq 0}x^\pm_{i,r}z^{-r-1}$, relation (L5) says that the coefficient of $z^{-r-1}w^{-s-1}$ in $(z-w)^{-a_{ij}}[x^\pm_i(z),x^\pm_j(w)]$ vanishes for all $r,s\geq 0$.  By (L4) the operators $\operatorname{ad}x^\pm_{i,r}$ commute with each other, so the order of $r_1,\dots,r_m$ in (L6) does not matter.

\begin{remark}\label{rem:loop}
  This is the Lie algebra $L\mathfrak{g}$ of \cite[Def.~3.24]{shanCenterQuiver2017}.  If $\Gamma$ is of finite type, it is the current algebra $\mathfrak{g}[t]$ \cite[Th.~3.1]{beliakovaCurrentAlgebras2017}, but this fails outside of finite type.  For the cyclic quiver with $e\geq 4$ vertices, $U(L\mathfrak{g})$ is Guay's affine Yangian at $\lambda=\beta=0$ \cite[Def.~2.3, \S 2]{guayAffineYangians2007}, which is the enveloping algebra of the universal central extension of $\mathfrak{sl}_e[u^{\pm1},v]$.
\end{remark}

For $n\geq 1$ and $f=\sum_{\mathbf{a}}f_{\mathbf{a}}x_1^{a_1}\cdots x_n^{a_n}\in\SymR_n$, let
\[x^{\pm(n)}_i[f]=\frac{1}{n!}\sum_{\mathbf{a}}f_{\mathbf{a}}\,x^\pm_{i,a_1}\cdots x^\pm_{i,a_n}\in U(L\mathfrak{g}).\]
By relation (L4), if $f$ is the monomial symmetric function with $k_r$ variables raised to the power $r$ for each $r$, then $x^{\pm(n)}_i[f]=\prod_r (x^\pm_{i,r})^{k_r}/k_r!$.
\begin{definition}\label{def:U-Z}
For each $n\geq 1$ and $i\in I$, let $\mathbf{X}^{\pm}_{n,i}$ be a copy of the abelian group $\SymR_n$, whose element corresponding to $f\in\SymR_n$ we denote by the formal symbol $\mathbf{x}^{\pm(n)}_i[f]$.
Let $\mathcal{F}$ be the free associative ring (tensor algebra) on the abelian group $\bigoplus_{n,i}\mathbf{X}^{\pm}_{n,i}$, and let $\mathcal{F}_\kk=\mathcal{F}\otimes\kk$.

  Let $U_\Z\subset U(L\mathfrak{g})$ be the image of the map $\mathcal{F}\to U(L\mathfrak{g})$ sending each symbol $\mathbf{x}^{\pm(n)}_i[f]$ to $x^{\pm(n)}_i[f]$. For a field $\kk$, we let $U_\kk=U_\Z\otimes\kk$.
\end{definition}

Note that $U_\Q=U_\Z\otimes\Q$ is $U(L\mathfrak{g})$ itself: the map $U_\Z\otimes\Q\to U(L\mathfrak{g})$ is injective since $U_\Z$ is a subgroup of the $\Q$-vector space $U(L\mathfrak{g})$, and surjective since its image contains the elements $x^\pm_{i,r}=x^{\pm(1)}_i[x_1^r]$, which generate $U(L\mathfrak{g})$ by (L3).  Since $U_\Z$ is the image of $\mathcal{F}$, if $U(L\mathfrak{g})$ acts on $M\otimes\Q$ for a torsion-free abelian group $M$, and the induced action of $\mathcal{F}$ preserves $M$, then this action factors through $U_\Z$, and so induces an action of $U_\kk$ on $M\otimes\kk$.

The ring $\mathcal{F}_\kk$ acts on three modules, compatibly with base change in $\kk$:
\begin{itemize}
\item on $H^*(\fM_{\Lambda};\kk)$, with $\mathbf{x}^{+(n)}_k[f]$ acting by $\sigma^+(\mu,k,n)E^{(n)}_k[f]$ and $\mathbf{x}^{-(n)}_k[f]$ by $\sigma^-(\mu,k,n)F^{(n)}_k[f]$ on the summand of weight $\mu$;
\item on $Z(\mathsf{R}^\Lambda(\kk))$, with $\mathbf{x}^{+(n)}_k[f]$ and $\mathbf{x}^{-(n)}_k[f]$ acting by $\mathsf{E}^{(n)}_k[f]\star-$ and $\mathsf{F}^{(n)}_k[f]\star-$, using the $\beta'$-Cautis--Lauda duality;
\item on $\Tr(\mathsf{R}^\Lambda(\kk))$, with the same elements acting by $[g]\mapsto[f(y)e_n\circ \mathcal{E}^n_k(g)]$, respectively $[g]\mapsto[f(y)e_n\circ \mathcal{F}^n_k(g)]$.  This action of $\Tr(\tU^*)$ on the cocenter is induced by horizontal composition \cite[\S 2.6]{beliakovaCurrentAlgebras2017}.
\end{itemize}

Fix the $\beta'$-Cautis--Lauda duality.  For a 1-morphism $v\colon\nu\to\mu$ of $\tU$ and a module $Q$ over the summand $\Rn[\alpha'](\kk)$ with $\nu=\Lambda-\alpha'$, let
\[\operatorname{ptr}_v\colon \operatorname{End}(v(Q))\to\operatorname{End}(Q),\qquad h\mapsto \epsilon'_v\circ v^\star(h)\circ\iota_v,\]
where $\iota_v\colon\mathrm{id}_\nu\to v^\star v$ and $\epsilon'_v\colon v^\star v\to\mathrm{id}_\nu$ are the unit and counit of the duality; diagrammatically, $\operatorname{ptr}_v$ closes off the strands of $v$ to the right of the green string $Q$:
\[\operatorname{ptr}_v(h)\;=\;
\begin{tikzpicture}[baseline=(box.center),very thick,scale=.8]
  \draw[green!70!black] (-1.6,-1.6) node[below] {$Q$} -- (-1.6,1.6);
  \draw (0,-.35) -- (0,-.9) arc (180:360:.6 and .5) -- (1.2,.9) arc (0:180:.6 and .5) -- (0,.35);
  \node[draw,rectangle,fill=white,minimum width=.9cm,minimum height=.7cm] (box) at (0,0) {$h$};
  \node[left] at (0,1.1) {$v$};
  \node[right] at (1.2,0) {$v^\star$};
  \node[above] at (.6,1.4) {$\epsilon'_v$};
  \node[below] at (.6,-1.4) {$\iota_v$};
\end{tikzpicture}\]

Here the module $Q$ is drawn as a green string on the left-hand boundary; this diagram gives an endomorphism of $Q$ by using bubble slides to drag the bubble past the strands of the cyclotomic quotient acting on $Q$, and then setting positive degree bubbles at far left to zero.
Since the duality is a biadjunction compatible with composition, $\operatorname{ptr}_{v}\circ\operatorname{ptr}_{w}=\operatorname{ptr}_{wv}$.
Let $\mu:=\Lambda-\alpha$. Define
\[t_\mu\colon \Rn[\alpha](\kk)\to\kk,\qquad t_\mu(h)=\sum_{\mathbf{i}}\operatorname{ptr}_{\mathcal{F}_{\mathbf{i}}}\big(e(\mathbf{i})he(\mathbf{i})\big)\in\operatorname{End}(\kk)=\kk.\]
Let $\langle z,[g]\rangle_\mu=t_\mu(zg)$ for $z\in Z(\Rn[\alpha](\kk)), [g]\in \Tr(\Rn[\alpha](\kk))$.
\begin{lemma}\label{lem:adjunction}\hfill
  \begin{enumerate}
  \item The form $t_\mu$ is a symmetrizing form on $\Rn[\alpha](\kk)$, so $\langle-,-\rangle_\mu$ is a perfect pairing between $Z(\Rn[\alpha](\kk))$ and $\Tr(\Rn[\alpha](\kk))$.
  \item For $u\colon\mu\to\nu$, $a\in\operatorname{End}(u)$, $z\in Z(\Rn[\alpha](\kk))$ and $f\in\Tr(\Rn[\alpha'](\kk))$ with $\nu=\Lambda-\alpha'$, we have $\langle a\star z,f\rangle_\nu=\langle z,[a^\star]\cdot f\rangle_\mu$.
  \end{enumerate}
\end{lemma}
\begin{proof}
  (1) The form $t_\mu$ differs from the Frobenius trace of \cite[Th.~3.18]{Webmerged} only by the signs of the twisted duality, which are constant on each summand $e(\mathbf{i})\Rn[\alpha] e(\mathbf{j})$, so it is again a Frobenius form.  It is symmetric because the $\beta'$-Cautis--Lauda duality is cyclic; compare \cite[Rem.~3.19]{Webmerged}.

  For (2), write $f=[g]$ with $g\in\operatorname{End}(Q)$ for a projective module $Q$ of weight $\nu$.  The closed diagram computing $t_\nu\big((a\star z)g\big)$ can be rewritten as:
  \[\begin{tikzpicture}[baseline=(az.center),very thick]
    \draw[green!70!black] (0,-1.4) -- (0,1.4) arc (180:0:.9 and .8) -- (1.8,-1.4) arc (0:-180:.9 and .8);
    \node[green!70!black,right] at (1.8,0) {$Q$};
    \node[draw,fill=white,minimum width=.7cm,minimum height=.6cm] at (0,.85) {$g$};
    \node[draw,fill=white,minimum height=.6cm] (az) at (0,0) {$a\star z$};
  \end{tikzpicture}
  \quad=\quad
  \begin{tikzpicture}[baseline=(z.center),very thick]
    \draw[green!70!black] (0,-1.4) -- (0,1.4) arc (180:0:2 and .8) -- (4,-1.4) arc (0:-180:2 and .8);
    \node[green!70!black,right] at (4,0) {$Q$};
    \draw (1.3,-1.4) -- (1.3,1.4) arc (180:0:.55 and .3) -- (2.4,-1.4) arc (0:-180:.55 and .3);
    \node[right] at (2.4,0) {$u$};
    \node[right] at (1.3,-.95) {$u^\star$};
    \node[draw,fill=white,minimum width=.7cm,minimum height=.6cm] (a) at (1.3,.85) {$a^\star$};
    \node[draw,fill=white,minimum width=.7cm,minimum height=.6cm] (g) at (0,.85) {$g$};
    \node[draw,fill=white,minimum width=2.3cm,minimum height=.6cm] (z) at (.65,0) {$z$};
    \draw[dashed,rounded corners=6pt,thin] (.65,1.9) -- (3,1.9) -- (3,-1.9) -- (.65,-1.9) -- (.65,-.45) -- (-.75,-.45) -- (-.75,.45) -- (.65,.45) -- cycle;
    \node[left,font=\small] at (-.8,-.45) {$a\star z$};
  \end{tikzpicture}\]
  Here the coupon $z$ stands for the action of the center $Z(\Rn[\alpha](\kk))$ on $u^\star Q$, and the coupon $a^\star$ for the induced endomorphism of this module; by cyclicity of the duality, the coupon $a^\star$ on the strand $u^\star$ is the coupon $a$ slid around the cap onto the strand $u$.  So, if we evaluate the portion of the diagram inside the dashed boundary, we obtain the action of $a\star z$ on $Q$, giving the equality of the two diagrams.  The right-hand side is the definition of $\langle z,[a^\star]\cdot f\rangle_\mu$, so the claim follows.
\end{proof}

\begin{proposition}\label{prop:loop-action}\hfill
  \begin{enumerate}
  \item For every field $\kk$, the operators $\mathsf{E}_{i,r}$ and $\mathsf{F}_{i,r}$, acting as $x^+_{i,r}$ and $x^-_{i,r}$, satisfy relations (L1)--(L6) over $\kk$ on $\Tr(\mathsf{R}^\Lambda(\kk))$ and on $Z(\mathsf{R}^\Lambda(\kk))$.
  \item The action of $\mathcal{F}$ on $Z(\mathsf{R}^\Lambda)$ factors through $U_\Z$.  Hence $U_\kk$ acts on $Z(\mathsf{R}^\Lambda)\otimes\kk$ for every field $\kk$.
  \end{enumerate}
\end{proposition}

\begin{proof}
  (1) For $\Tr(\mathsf{R}^\Lambda(\kk))$, this is \cite[Th.~3.25]{shanCenterQuiver2017}.

  For the center, the existence of this action is noted in \cite[Rem. 3.41]{shanCenterQuiver2017}, but with slightly different conventions, and an implicit assumption of finite type and characteristic 0.  More explicitly,
  since the $\beta'$-Cautis--Lauda duality is cyclic \cite{Webcenter}, $[a]\mapsto[a^\star]$ is an anti-involution of $\Tr(\tU^*)$.  It exchanges $\mathsf{E}_{i,r}$ and $\mathsf{F}_{i,r}$, since the dot is self-dual up to sign.
By \cref{lem:adjunction}, $[a]\star-$ is the transpose of the action of $[a^\star]$ on the cocenter, under a perfect pairing.  The anti-involution $x^\pm_{i,r}\mapsto x^\mp_{i,r}$, $\xi_{i,r}\mapsto\xi_{i,r}$ of $L\mathfrak{g}$ preserves the relations so they hold on the center as well.

  (2) By (1) with $\kk=\Q$, $U(L\mathfrak{g})$ acts on $Z(\mathsf{R}^\Lambda(\Q))=Z(\mathsf{R}^\Lambda)\otimes\Q$, the equality holding since the center commutes with flat base change.  By \cref{lem:KLR-divided}, applied to $\mathsf{E}$ and $\mathsf{F}$, and since convolution is an action of $\Tr(\tU^*)$, the element $x^{\pm(n)}_k[f]$ acts by $\mathsf{E}^{(n)}_k[f]\star-$, respectively $\mathsf{F}^{(n)}_k[f]\star-$, which preserves $Z(\Rn[\alpha])$.  The group $Z(\Rn[\alpha])$ is torsion-free, since it is a subgroup of $\Rn[\alpha]$, which is free by \cref{lem:free}.  Thus, we obtain an action of $U_{\Z}$ which base-changes to an action of $U_{\kk}$.
\end{proof}
\subsection{The integral map and injectivity}
\label{sec:injective}

Let $S=H^*_{G_{\Bv}}(\mathrm{pt};\Z)=\bigotimes_i\Z[c_{i;1},\dots,c_{i;v_i}]$ be the equivariant cohomology of a point, with $c_{i;j}$ the $j$th Chern class of the tautological bundle on $BG_{\Bv}$.
Let $\kappa_g\colon S\to H^*(\fM(\Bv,\Bw);\Z)$ be the Kirwan map, sending $c_{i;j}$ to the $j$th Chern class of the tautological bundle $V_i$, and let $\kappa_a\colon S\to Z(\Rn[\alpha])$ be the algebra homomorphism sending $c_{i;j}$ to $e_{i;j}$.  Let $\kappa_{g,\C}\colon S\otimes_{\Z}\C\to H^*(\fM(\Bv,\Bw);\C)$ be the composition of the complexification of $\kappa_g$ with the natural isomorphism $H^*(\fM(\Bv,\Bw);\Z)\otimes_{\Z}\C\cong H^*(\fM(\Bv,\Bw);\C)$. Let
$\kappa_{a,\C}\colon S\otimes_{\Z}\C\to Z(\Rn[\alpha](\C))$ be the composition of the complexification of $\kappa_a$ with the natural isomorphism $Z(\Rn[\alpha])\otimes\C\cong Z(\Rn[\alpha](\C))$.
\begin{theorem}[\mbox{\cite[Th.~1.2]{mcgertyKirwanSurjectivity2018}}]\label{thm:MN}
  The map $\kappa_g$ is surjective.
\end{theorem}

\begin{lemma}\label{lem:phi-Z}
  There is a unique ring homomorphism $\phi_\Z\colon H^*(\fM_{\Lambda,\alpha};\Z)\to Z(\Rn[\alpha])$ with $\phi_\Z(c_{i;j})=e_{i;j}$.  It is injective, and $\phi_\Z\otimes \C$ is the map $\phi$ of \cite{Webcenter}.  It is a homomorphism of $\mathcal{F}$-modules, for the actions described after \cref{def:U-Z}.
\end{lemma}
\begin{proof}
 By \cref{thm:MN}, to construct a homomorphism, it suffices to check that the kernel of $\kappa_g$ is contained in the kernel of $\kappa_a$.  By \cite[Prop.~3.5]{Webcenter}, there is a homomorphism $\phi_{\C}\colon H^*(\fM_{\Lambda};\C)\to Z(\mathsf{R}^\Lambda(\C))$ such that
 \[\phi_{\C}\circ\kappa_{g,\C}=\kappa_{a,\C}.\]
By \cref{lem:pairing} and \cref{lem:free}, the maps from $H^*(\fM_{\Lambda};\Z)$ and $\mathsf{R}^\Lambda$ to their complexifications are injective.  Consequently,
 \[\ker\kappa_g=S\cap\ker\kappa_{g,\C}\subset S\cap\ker\kappa_{a,\C}=\ker\kappa_a.\]
 Thus $\phi_\Z$ exists and is the restriction of $\phi_{\C}$ to the integral lattices.

 Injectivity follows from that of $\phi$ \cite[Th.~3.4]{Webcenter}.  Finally, the operators by which the generators of $\mathcal{F}$ act preserve both lattices, and $\phi$ intertwines them by \cref{lem:intertwine}.
\end{proof}

For a field $\kk$, let $\phi_{\kk}$ be the composition
\[H^*(\fM_{\Lambda,\alpha};\kk)=H^*(\fM_{\Lambda,\alpha};\Z)\otimes\kk\xrightarrow{\phi_\Z\otimes\kk}Z(\Rn[\alpha])\otimes\kk\to Z(\Rn[\alpha](\kk)),\]
where the last map is induced by $\Rn[\alpha]\otimes\kk=\Rn[\alpha](\kk)$; it lands in the center since the image of a central element is central.  This is an algebra homomorphism with $\phi_{\kk}(c_{i;j})=e_{i;j}$, and it is a homomorphism of $\mathcal{F}_\kk$-modules, since $\phi_\Z$ is a homomorphism of $\mathcal{F}$-modules.

\begin{proof}[Proof of \cref{thm:A}(1)--(2)]
  Existence, uniqueness, and the comparison with $\phi$ are \cref{lem:phi-Z}.  It remains to show that $\phi_{\kk}$ is injective.

  By \cref{lem:phi-Z}, $\phi_{\kk}$ intertwines the actions of $\mathcal{F}_\kk$ on $H^*(\fM_{\Lambda};\kk)$ and $Z(\mathsf{R}^\Lambda(\kk))$.  In particular, its kernel is a graded $\mathcal{F}_\kk$-submodule.  By \cref{lem:cyclic-cohomology}, any nonzero such submodule contains $H^*(\fM_{\Lambda,0};\kk)$.  The restriction of $\phi_{\kk}$ to this one-dimensional highest-weight summand sends $1$ to $1$, so its kernel is zero.

  By \cref{thm:MN}, the classes $c_{i;j}$ generate $H^*(\fM_{\Lambda,\alpha};\kk)$, so the image of $\phi_{\kk}$ is exactly the JM center.  This proves (1).  Since $\phi_{\kk}$ is injective, we have
  \[\dim H^*(\fM_{\Lambda,\alpha};\kk)\leq \dim Z(\Rn[\alpha](\kk)).\]
  It is an isomorphism if and only if equality holds, which is equivalent to the reverse inequality in (2); in that case the full center coincides with the JM center.  This proves (2).

Note that this establishes that the action of $\mathcal{F}_\kk$ on $H^*(\fM_{\Lambda};\kk)$ factors through $U_\kk$: the map $\phi_\Z\otimes\kk\colon H^*(\fM_{\Lambda};\kk)\to Z(\mathsf{R}^\Lambda)\otimes\kk$ is a homomorphism of $\mathcal{F}_\kk$-modules, it is injective since $\phi_\kk$ factors through it, and the action on $Z(\mathsf{R}^\Lambda)\otimes \kk$ factors through $U_\kk$ by \cref{prop:loop-action}.
\end{proof}

\begin{corollary} Let $\Gamma$ be a finite graph without edge loops and $\mathfrak{g}$ be the corresponding symmetric Kac--Moody Lie algebra. Let $\Rn[\alpha](\Z)$ be the cyclotomic KLR algebra over $\Z$ associated with $\mathfrak{g}$. Then the Jucys--Murphy center of $\Rn[\alpha](\Z)$ is a pure $\Z$-submodule of $\Rn[\alpha](\Z)$.
\end{corollary}
\begin{proof} This follows from \cref{thm:A}(1) and (\ref{eq:UCT}).
\end{proof}

The injectivity in \cref{thm:A}(1) also allows us to pass from a
characteristic-zero isomorphism to an integral isomorphism.

\begin{proposition}\label{prop:integral-comparison}
Let $\Gamma$ be a finite graph without edge loops, let $\Lambda$ be
a dominant weight, and let $\alpha\in Q^+$.
If $\phi_{\kk}$ is an isomorphism for some field $\kk$ of
characteristic zero, then
\[
  \phi_{\Z}\colon H^*(\fM_{\Lambda,\alpha};\Z)
  \longrightarrow Z(\Rn[\alpha](\Z))
\]
is an isomorphism of rings and $Z(\Rn[\alpha](\Z))=Z(\Rn[\alpha](\Z))^{\rm{JM}}$.
\end{proposition}

\begin{proof}
By \cref{lem:phi-Z}, $\phi_{\Z}$ is injective.
The center of a finite-dimensional algebra commutes with extension
of the ground field. Together with \cref{lem:pairing}(3) and the
base-change compatibility of the construction in \cref{lem:phi-Z},
this identifies $\phi_{\kk}$ with
$\phi_{\Q}\otimes_{\Q}\kk$.
Since $\kk$ is faithfully flat over $\Q$, the hypothesis implies
that $\phi_{\Q}$ is an isomorphism.

Let $z\in Z(\Rn[\alpha](\Z))$.
Surjectivity of $\phi_{\Q}$, followed by clearing rational
denominators, gives a positive integer $d$ and a class
$x\in H^*(\fM_{\Lambda,\alpha};\Z)$ such that
\begin{equation}\label{eq:integral-comparison-multiple}
  dz=\phi_{\Z}(x).
\end{equation}
This equality holds over $\Z$, since $\Rn[\alpha](\Z)$ is free
abelian by \cref{lem:free} and hence embeds into
$\Rn[\alpha](\Q)$.
Choose $d$ minimal among positive integers satisfying
\eqref{eq:integral-comparison-multiple} for some such $x$,
and fix a corresponding $x$.

Suppose that $d>1$, and let $p$ be a prime divisor of $d$.
Reducing \eqref{eq:integral-comparison-multiple} modulo $p$ gives
$\phi_{\mathbb F_p}(\bar x)=0$, where $\bar x$ is the reduction
of $x$ modulo $p$.
By \cref{thm:A}(1), $\bar x=0$.
By \cref{lem:pairing}(3), reduction modulo $p$ identifies
$H^*(\fM_{\Lambda,\alpha};\mathbb F_p)$ with
\[
  H^*(\fM_{\Lambda,\alpha};\Z)/
  pH^*(\fM_{\Lambda,\alpha};\Z).
\]
Thus $x=px'$ for some
$x'\in H^*(\fM_{\Lambda,\alpha};\Z)$.
Substituting into \eqref{eq:integral-comparison-multiple} gives
\[
  p\left(\frac{d}{p}z-\phi_{\Z}(x')\right)=0.
\]
By \cref{lem:free}, $\Rn[\alpha](\Z)$ is torsion-free, so
\[
  \frac{d}{p}z=\phi_{\Z}(x').
\]
This contradicts the minimality of $d$.
Hence $d=1$, and $z$ lies in the image of $\phi_{\Z}$.
Therefore $\phi_{\Z}$ is an isomorphism.
Finally, its image is the Jucys--Murphy center by the construction
in \cref{lem:phi-Z}, giving the asserted equality of centers.
\end{proof}

\section{Finite and affine type A}
\label{sec:affineA}

To apply \cref{thm:A}(1)--(2), we need an upper bound on the dimension
of $Z(\Rn[\alpha](\kk))$, which is not known in general.  In finite and
affine type A, the required bound follows from the dimension theorem
\cref{thm:HS} below.

Let us recall the definition of non-degenerate cyclotomic Hecke algebras.
\begin{definition}[{\cite[Def.~3.1]{ArikiKoike94}}]\label{def:Hecke}
  Let $\kk$ be a commutative ring, and let $q\in\kk^\times$ and $\bQ=(Q_1,\dots,Q_r)\in(\kk^\times)^r$.  The \emph{cyclotomic Hecke algebra} (or \emph{Ariki--Koike algebra}) $\Hn=\Hn(\kk)$ of type $G(r,1,n)$, with Hecke parameter $q$ and cyclotomic parameters $\bQ$, is the unital $\kk$-algebra generated by $T_0,T_1,\dots,T_{n-1}$ subject to the relations
  \begin{align*}
    &\prod_{m=1}^r(T_0-Q_m)=0, && (T_a-q)(T_a+1)=0 \quad(1\leq a<n),\\
    &T_0T_1T_0T_1=T_1T_0T_1T_0, && T_aT_b=T_bT_a\quad(0\leq a<b-1<n-1),\\
    &T_aT_{a+1}T_a=T_{a+1}T_aT_{a+1}\quad(1\leq a<n-1).
  \end{align*}
  The \emph{Jucys--Murphy elements} of $\Hn$ are $L_1=T_0$ and $L_{a+1}=q^{-1}T_aL_aT_a$ for $1\leq a<n$.
\end{definition}
\begin{lemma}[{\cite[\S 3]{ArikiKoike94}}]\label{L:JMCenter}
  The Jucys--Murphy elements $L_1,\dots,L_n$ commute with each other, and every symmetric polynomial in them is central in $\Hn$.
\end{lemma}

We call the subalgebra $Z(\Hn)^\JM\subseteq Z(\Hn)$ consisting of the symmetric polynomials in $L_1,\dots,L_n$ the \emph{Jucys--Murphy center} of $\Hn$.  Both $Z(\Hn)$ and $Z(\Hn)^\JM$ commute with extension of scalars from a field $\kk$ to a larger field, so the same is true of the equality $Z(\Hn)=Z(\Hn)^\JM$.  (In particular, \cref{thm:B} also covers the version of $\Hn$ defined by a relation $P(T_0)=0$ for a monic polynomial $P$ with $P(0)\neq 0$ that does not split over $\kk$.)

We first recall the extension of Brundan--Kleshchev's isomorphism between cyclotomic KLR algebras and cyclotomic Hecke algebras to cyclotomic parameters lying in several $q$-orbits.
\begin{theorem}[{\cite[Main Theorem]{BKKL}, \cite[Th.~1.6]{WebBKnote}, \cite[Th.~3.4]{Ros}}]\label{thm:BK}
  Let $\kk$ be a field, $q\in\kk^\times$ with $q\neq 1$, and $\bQ\in(\kk^\times)^r$.  Then $\Hn(\kk)$ is isomorphic to the cyclotomic KLR algebra $\Rn(\kk)$ over $\kk$ of a graph $\Gamma=(I,E)$ which is a union of quivers of type $A_\infty$ or $A^{(1)}_{e-1}$, for some $e\geq 2$, where
  \begin{equation}\label{eq:I-E}
    I=\{q^aQ_m\mid a\in \Z,\ 1\leq m\leq r\}\subset \kk^\times,\qquad E=\{(x,qx)\mid x\in I\},
  \end{equation}
  and $\Lambda=\sum_{i\in I}w_i\Lambda_i$ with $w_i=\#\{m\mid Q_m=i\}$.
\end{theorem}
If $q$ has finite multiplicative order $e$, the $q$-orbits in $I$ give the $e$-cycle components of $\Gamma$; if $q$ has infinite order, they give the $A_\infty$ components.  The decomposition of $I$ into connected components will be familiar to readers versed in the representation theory of Ariki--Koike algebras: having $I$ connected is precisely the $q$-connectedness condition of \cite{dipperMoritaEquivalences2002}, and the KLR algebra makes the Morita equivalence of \cite[Th.~1.2]{dipperMoritaEquivalences2002} with a sum of tensor products of Hecke algebras with $q$-connected parameters explicit.

We identify an $r$-multipartition $\blam=(\lambda^{(1)}|\dots|\lambda^{(r)})$ of $n$ with its diagram, and give the node in row $a$ and column $c$ of the $m$th component the \emph{residue} $Q_mq^{c-a}\in I$.  Let $\Parts$ be the set of $r$-multipartitions of $n$, and for $\alpha=\sum_iv_i\alpha_i\in Q^+_n$, let $\Parts[\alpha]$ be the set of $\blam\in\Parts$ having exactly $v_i$ nodes of residue $i$ for every $i\in I$; thus $\Parts=\bigsqcup_{\alpha\in Q^+_n}\Parts[\alpha]$.

\begin{lemma}\label{lem:dim-H}
  For any field $\kk$ and the graph $\Gamma$ of \cref{thm:BK}, we have $\dim H^*(\fM_{\Lambda,\alpha};\kk)=|\Parts[\alpha]|$.
\end{lemma}
\begin{proof}
  Since the quiver variety for a union of quivers is the product of those for its components, we may assume $\Gamma$ is connected.  For a cyclic or linear quiver, the torus $T=(\C^\times)^2\times T_{\Bw}$ acting on $\fM_{\Lambda,\alpha}$ through the plane $\C^2$ and the framing has isolated fixed points, in bijection with $\Parts[\alpha]$ \cite[Prop.~3.8]{fujiiCombinatorialStudy2017}.  Note that Fujii--Minabe's approach only applies precisely as written if $I=\mu_e$ is the $e$th roots of unity, in which case, our residue $Q_mq^{c-a}$ is the scalar which $q\in \mu_e$ acts by the corresponding representation.  Of course, as a geometric statement about quiver varieties, this case is perfectly adequate, since if $\Gamma$ is a union of $e$-cycles, the quiver variety is a product of quiver varieties for the $e$-cycle, and if $\Gamma$ is a union of $A_\infty$ quivers, the quiver variety is a product of quiver varieties for $e$-cycles for $e$ sufficiently large.  We leave the verification that our notion of residue works as desired in this case to the reader.

  The Bia\l{}ynicki-Birula decomposition for an appropriately chosen cocharacter of $T$ partitions $\fM(\Bv,\Bw)$ into affine cells indexed by the fixed points \cite[Lem.~7.1.5]{nakajimaQuiverVarieties2001a}.  Thus $H^*(\fM_{\Lambda,\alpha};\Z)$ is free of rank $|\Parts[\alpha]|$, so $\dim H^*(\fM_{\Lambda,\alpha};\kk)=|\Parts[\alpha]|$ for every field $\kk$.
\end{proof}

We use the following dimension theorem.
\begin{theorem}[{\cite[Th.~1.4]{huCocenterCyclotomic2022}, \cite[Th.~1]{Brundan:degenCentre}}]\label{thm:HS}
  Let $\kk$ be a field and $q\neq 1$.  Then $\dim Z(\Hn(\kk))=|\Parts|$.  The same dimension formula holds for a degenerate cyclotomic Hecke algebra of level $r$.
\end{theorem}

\begin{proof}[Proof of \cref{thm:A}, part (3)]
  We may extend the field, since the relevant dimensions are unchanged by field extension.  If $p=0$ or $p\nmid e$, choose a primitive $e$th root of unity $q\in\kk$ and cyclotomic parameters $\bQ$ whose $q$-orbits realize the components of $\Gamma$ and the weight $\Lambda$.  Then \cref{thm:BK} identifies $\Rn(\kk)$ with $\Hn(\kk)$.  If $p=e$, the degenerate Brundan--Kleshchev isomorphism \cite{BKKL} gives the same identification with a degenerate cyclotomic Hecke algebra.

  By \cref{thm:A}(1), the direct sum
  \[\phi_{\kk}\colon \bigoplus_{\alpha\in Q^+_n}H^*(\fM_{\Lambda,\alpha};\kk)\longrightarrow \bigoplus_{\alpha\in Q^+_n}Z(\Rn[\alpha](\kk))=Z(\Rn(\kk))\]
  is injective.  By \cref{lem:dim-H}, its source has dimension $|\Parts|$, while \cref{thm:HS} and the preceding identifications show that its target has the same dimension.  Hence every summand $H^*(\fM_{\Lambda,\alpha};\kk)\to Z(\Rn[\alpha](\kk))$ is an isomorphism.
\end{proof}

Applying \cref{prop:integral-comparison} to the
characteristic-zero case of \cref{thm:A}(3) gives the following
integral isomorphism, with no restriction on the prime divisors
of $e$.

\begin{corollary}\label{cor:affine-A-integral-comparison}
Suppose that every connected component of $\Gamma$ is an
$e$-cycle, for the same $e\geq 2$.
Then, for every dominant weight $\Lambda$ and every $\alpha\in Q^+$,
\[
  \phi_{\Z}\colon H^*(\fM_{\Lambda,\alpha};\Z)
  \longrightarrow Z(\Rn[\alpha](\Z))
\]
is an isomorphism of rings.
\end{corollary}

\begin{proof}
By \cref{thm:A}(3), applied over $\Q$, the map $\phi_{\Q}$
is an isomorphism.
The conclusion follows from \cref{prop:integral-comparison}.
\end{proof}

Applying \cref{thm:A}(1) on the other side of the Brundan--Kleshchev isomorphism of \cref{thm:BK}, we obtain the proof of \cref{thm:B}.

\begin{proof}[Proof of \cref{thm:B}]

  Consider the subset $I\subset \kk^\times$ of \eqref{eq:I-E} and the graph $\Gamma$ of \cref{thm:BK}.  For $\alpha=\sum_iv_i\alpha_i\in Q^+_n$, let $\Hn[\alpha]$ be the summand of $\Hn$ on which the simultaneous generalized eigenvalues of $L_1,\dots,L_n$ form the unordered tuple in which each $i\in I$ occurs $v_i$ times; \cref{thm:BK} identifies $\Hn[\alpha]$ with $\Rn[\alpha]$, and $\Hn=\bigoplus_\alpha\Hn[\alpha]$.  Let $n_{i;j}\in Z(\Hn[\alpha])$ be the $j$th elementary symmetric polynomial in the nilpotent parts of those $L_a$ whose generalized eigenvalue is $i$.  The isomorphism of \cref{thm:BK} can be chosen so that $n_{i;j}$ is the image of $i^j e_{i;j}$; in the notation of \cite[Th.~3.10]{WebBKnote}, take $b(u)=1+u$, so that $L_ae(\mathbf{i})$ corresponds to $i_a(1+y_a)e(\mathbf{i})$.

  We now write each $n_{i;j}$ as a symmetric polynomial in the $L_a$.  Let $C\subset I$ be the set of eigenvalues of $L_1,\dots,L_n$ on $\Hn$, and choose $N$
such that $(L_a-i_a)^N e(\mathbf{i})=0$ for all $a$ and
$\mathbf{i}$.
For $i\in C$, let $\ell_i(T)=\prod_{c\in C\setminus\{i\}}\frac{T-c}{i-c}$ be the Lagrange interpolation polynomial, and set
\[
  p_i(T)=1-\bigl(1-\ell_i(T)^N\bigr)^N.
\]
Then
\[
  p_i(T)\equiv\delta_{ic}\pmod{(T-c)^N}
  \qquad(c\in C).
\]
Indeed, we have $1-\ell_i(T)^N\equiv 1 \pmod {(T-c)^N}$ if $c\neq i$; and since $T-i$ divides $1-\ell_i^N$, we have $(T-i)^N$ divides $(1-\ell_i^N)^N$.

Set $D_i(T)=(T-i)p_i(T)$. Then
\begin{equation}\label{eq:JM-colored-nilpotent-action}
  p_i(L_a)e(\mathbf{i})=
  \begin{cases}
   e(\mathbf{i}),&i_a=i,\\
    0,&i_a\neq i,
  \end{cases}\qquad   D_i(L_a)e(\mathbf{i})=
  \begin{cases}
    (L_a-i)e(\mathbf{i}),&i_a=i,\\
    0,&i_a\neq i.
  \end{cases}
\end{equation}
This means that $p_{i_1}(L_1)\cdots p_{i_n}(L_n)$ is the idempotent $e(\mathbf{i})$ for any $\mathbf{i}=(i_1,\dots,i_n)\in C^n$, and since $e(\mathbf{i})=0$ for $\mathbf{i}\notin C^n$, $\sum_{\mathbf{i}\in I^\alpha\cap C^n}p_{i_1}(L_1)\cdots p_{i_n}(L_n)=1_\alpha$ is manifestly symmetric and the block idempotent of $\Hn[\alpha]$.  In particular, $1_\alpha\in Z(\Hn)^\JM$, and $Z(\Hn)^\JM=\bigoplus_\alpha 1_\alpha Z(\Hn)^\JM$.

Under the above identification, for each $\mathbf{i}\in I^\alpha$,
\eqref{eq:JM-colored-nilpotent-action} gives
\begin{align}
  e_j\bigl(D_i(L_1),\ldots,D_i(L_n)\bigr)e(\mathbf{i})
  &=
  \sum_{\substack{1\leq a_1<\cdots<a_j\leq n\\
                   i_{a_1}=\cdots=i_{a_j}=i}}
  (L_{a_1}-i)\cdots(L_{a_j}-i)e(\mathbf{i})
  \notag\\
  &=i^j e_j(y_a\mid i_a=i)e(\mathbf{i})
  \notag\\
  &=n_{i;j}e(\mathbf{i}).
  \label{eq:JM-colored-symmetric-action}
\end{align}
The second equality uses
$(L_a-i)e(\mathbf{i})=iy_ae(\mathbf{i})$ when $i_a=i$,
and the last equality is the definition of $n_{i;j}$.
These equalities also hold when $e(\mathbf{i})=0$.

Since $\sum_{\mathbf{i}\in I^\alpha}e(\mathbf{i})=1$
in $\Hn[\alpha]$, summing
\eqref{eq:JM-colored-symmetric-action} gives
\begin{align}
  e_j\bigl(D_i(L_1),\ldots,D_i(L_n)\bigr)
  &=
  \sum_{\mathbf{i}\in I^\alpha}
  e_j\bigl(D_i(L_1),\ldots,D_i(L_n)\bigr)e(\mathbf{i})
  \notag\\
  &=i^j\sum_{\mathbf{i}\in I^\alpha}
  e_j(y_a\mid i_a=i)e(\mathbf{i})
  \notag\\
  &=i^j e_{i;j}=n_{i;j}.
  \label{eq:JM-colored-symmetric-elements}
\end{align}

So $n_{i;j}=e_j\bigl(D_i(L_1),\ldots,D_i(L_n)\bigr)1_\alpha$ lies in $Z(\Hn)^\JM$.

  Conversely, on a fixed summand $\Hn[\alpha]$ every symmetric polynomial in the $L_a$ is a polynomial in the scalars $i$ and the elements $n_{i;j}$.  Thus the subalgebra generated by the $n_{i;j}$ is exactly the Jucys--Murphy center $Z(\Hn[\alpha])^\JM$.
  We have a commutative diagram
  \[\begin{tikzcd}
     & S\ar[d,"\kappa_a"]\ar[dl,"\kappa_g"']\ar[dr,"\kappa_H"] & \\
  H^*(\fM(\Bv,\Bw);\kk)\ar[r,"\phi_\kk"] & Z(\Rn[\alpha](\kk)) \ar[r,"\sim"']& Z(\Hn[\alpha]),
  \end{tikzcd}\]
  where, by definition, the map $\kappa_H$ sends $c_{i;j}$ to $i^{-j}n_{i;j}$.  Since $\kappa_g$ is surjective and the bottom row is injective, the image of $\kappa_H$, which is $Z(\Hn[\alpha])^\JM$, has dimension $|\Parts[\alpha]|$ by \cref{lem:dim-H}.  If this image were proper for any $\alpha$, summing over all $\alpha\in Q^+_n$ would give
  \[\dim Z(\Hn)>\sum_{\alpha\in Q^+_n}|\Parts[\alpha]|=|\Parts|,\]
  contradicting \cref{thm:HS}.  Thus $\kappa_H$ is surjective for every $\alpha$, and $Z(\Hn)=Z(\Hn)^\JM$.
\end{proof}

Careful investigation of the proof of \cref{thm:B}
reveals that the proof does not use the classification of the blocks of
$\Hn$ from \cite{LyleMathas07}. In several places the proof uses
\cref{L:JMCenter}, but this follows easily from the relations.  Following \cite[Def.~3.10, Th.~3.26]{DJM:cyc}, $\Hn$ is a cellular algebra whose cell modules are the Specht modules $S^\blam$ for $\blam\in\Parts$.

\begin{corollary}[{Lyle--Mathas~\cite[Th.~A]{LyleMathas07}}]\label{C:blocks}
  Let $\kk$ be a field, $q\neq 1$ and $\bQ\in(\kk^\times)^r$.  For every $\alpha\in Q^+_n$ with $\Parts[\alpha]\neq\emptyset$, the summand $\Hn[\alpha]$ is a block of $\Hn$, and the Specht module $S^\blam$ lies in $\Hn[\alpha]$ if and only if $\blam\in\Parts[\alpha]$.  In particular, two Specht modules $S^\blam$ and $S^\bmu$ lie in the same block if and only if $\blam$ and $\bmu$ have the same residues, counted with multiplicity.
\end{corollary}
\begin{proof}
  We have now proven that $Z(\Hn[\alpha])\cong Z(\Rn[\alpha](\kk))^\JM$.  The latter ring is a quotient of a polynomial ring in positively-graded variables by a homogeneous ideal, so it is a positively-graded local ring.  This shows that $\Hn[\alpha]$ has no non-trivial central idempotents, that is, it is a block.  The match of Specht modules to blocks is just computing the generalized eigenvalues of each $L_a$ on a Specht module, for example, by \cite[Prop.~3.7]{JamesMathas:cycJantzenSum}.
\end{proof}
\begin{remark}\label{rem:p-divides-e}
  Suppose $p\mid e$ and $p\neq e$.  There is no cyclotomic Hecke algebra in characteristic $p$ whose blocks are cyclotomic KLR algebras of type $A^{(1)}_{e-1}$: the orbits of multiplication by $\xi\in\overline{\F}_p^\times$ have size prime to $p$, and those of $x\mapsto x+1$ have size $p$.  So \cref{thm:HS} is not available.  By \cref{thm:A}, $\phi_{\kk}$ is still injective, so $\dim Z(\Rn[\alpha](\kk))\geq |\Parts[\alpha]|$.  Since $\Rn[\alpha](\kk)$ is a symmetric algebra (\cref{lem:adjunction}), $\dim Z(\Rn[\alpha](\kk))=\dim\Tr(\Rn[\alpha](\kk))=\dim \Tr(\Rn[\alpha])\otimes\kk$, which is $|\Parts[\alpha]|$ plus the number of $p$-torsion summands of $\Tr(\Rn[\alpha])$.  Thus $\phi_{\kk}$ is an isomorphism if and only if $\Tr(\Rn[\alpha])$ has no $p$-torsion.
\end{remark}

\section{The center and \texorpdfstring{$K$}{K}-theory of Gieseker
varieties}\label{SS:KTheory}

This final section proves \cref{thm:C}, which connects the $K$-theory
of the quiver varieties with the centers of the cyclotomic Hecke algebras,
building on ideas in \cite{KrylovPaegelowShlykov26,Paegelow26}.

We begin with an equivariant form of Kirwan surjectivity in $K$-theory, for any quiver variety $\fM=\fM(\Bv,\Bw)$ as in \cref{sec:notation}, but allowing edge loops.  The quiver variety $\fM$ has an induced action of $G_\Bw=\prod_iGL(W_i)$, and a commuting action of $(\Gm)^{E}$ acting by scaling each arrow $h$ by the corresponding scalar, and $\bar{h}$ by its inverse.  Finally, there is another action of $\mathbb{S}=\Gm$ by the scalar multiplication on the vector space \[\bigoplus_{h\in E\cup \bar{E}} \Hom(V_{\mathrm{out}(h)},V_{\mathrm{in}(h)})\oplus\bigoplus_{i\in I}\big(\Hom(V_i,W_i)\oplus \Hom(W_i,V_i)\big).\] We'll be interested in a torus $T$ acting on $\fM$ by a homomorphism $T\to G_\Bw\times (\Gm)^{E}\times \mathbb{S}$.  We can then consider the tautological bundles $\mathcal{V}_i$ on $\fM$ as $T$-equivariant vector bundles, by taking the trivial vector bundle with fiber $V_i$ before reduction with trivial $T$-action, and modding out by the diagonal action of $G_\Bv$; we call this the {\bf standard action} on this vector bundle.

For a torus $T$ acting on a variety $X$, write $\KC_T(X)$ for the Grothendieck group of $T$-equivariant coherent sheaves on $X$.

\begin{lemma}\label{lem:Kirwan-K}
  For $T$ as above, the Kirwan map $R(G_\Bv\times T)\to\KC_T(\fM)$, sending a representation to the associated vector bundle, is surjective.  In particular, taking $T$ trivial, the Grothendieck group $K(\fM)$ of \eqref{eq:ordinary-K-definition} is generated as a ring by the classes $[\Lambda^k\mathcal{V}_i]$ and $[\Lambda^{v_i}\mathcal{V}_i]^{-1}$, for $i\in I$ and $1\leq k\leq v_i$.
\end{lemma}

\begin{proof}
  Without the torus, the first statement follows from \cite[Th.~1.6(2)]{mcgertyKirwanSurjectivity2018}: $D^b\mathrm{Coh}(\fM)$ is generated, using direct sums, shifts and cones, by finitely many tautological bundles.  We check that the proof is $T$-equivariant, as McGerty and Nevins anticipate \cite[p.~163]{mcgertyKirwanSurjectivity2018}, and we use their notation.

  The image of $T$ in $G_\Bw$ is a torus, so we may choose bases of the $W_i$ in which it acts diagonally.  Regarding the components of $i$ and $j$ in these bases as arrows from and to an extra vertex $\infty$ gives the Crawley-Boevey quiver $Q$ of \cite[\S3.2]{mcgertyKirwanSurjectivity2018}, with dimension vector $\alpha$ equal to $\Bv$ on $I$ and to $1$ at $\infty$, and $G(\alpha)\cong G_\Bv$.  Let $\eta\colon T\to\mathbb{S}$ be the last component of the homomorphism $T\to G_\Bw\times(\Gm)^E\times\mathbb{S}$.  Then $T$ scales each arrow $h$ of the doubled quiver by a character $\chi_h$, with $\chi_h\chi_{\bar h}=\eta^2$: the factors $G_\Bw$ and $(\Gm)^E$ scale $h$ and $\bar h$ by inverse characters, and $\mathbb{S}$ scales every arrow by $\eta$.

  First, we note that the compactification $\overline{\fM}$ is $T$-equivariant.  Let $t\in T$ act on $\operatorname{Rep}(Q^{\mathrm{gtr}},\alpha^{\mathrm{gtr}})$ by $a_{h,n}\mapsto\chi_h(t)a_{h,n}$ and $e_{i,n}\mapsto e_{i,n}$.  This preserves the relations of \cite[Def.~4.3]{mcgertyKirwanSurjectivity2018}: it multiplies the preprojective relations by $\eta(t)^2$, and the relations $ea=ae$ are linear in the $a_{h,n}$.  It commutes with $G^{\mathrm{gtr}}$, so it preserves the $\theta^{\mathrm{gtr}}$-stable locus.  Hence $T$ acts on the compactification $\overline{\fM}$ and on the closure of $\fM$ in it, and $\iota=\mathrm{Ind}\colon\fM\to\overline{\fM}$ is $T$-equivariant.

  Second, consider the resolution of the graph $\Delta_\iota\subset\fM\times\overline{\fM}$ of $\iota$.  Let $T$ act trivially on $L(V_{\bullet,0},W_{\bullet,0})$, by $\chi_h$ on the summand of $E(V_{\bullet,0},W_{\bullet,1})$ indexed by $h$, and by $\eta^2$ on $L(V_{\bullet,0},W_{\bullet,2})$.  With this choice of action, the maps $\partial_0,\partial_1$ of \cite[(5.3)]{mcgertyKirwanSurjectivity2018} are $T$-equivariant: the $h$-component of $\partial_0(\phi)$ is linear in arrows of weight $\chi_h$, and each term of $\partial_1(\psi)$ pairs a component $\psi_h$ with an arrow of weight $\chi_{\bar h}$, so it has weight $\eta^2$.  The section $s$ is $\partial_0$ of the invariant section $1\in\Hom(V_{\infty,0},W_{\infty,0})$, so it is $T$-invariant, and the Koszul resolution \cite[(6.2)]{mcgertyKirwanSurjectivity2018} of $\mathcal O_{\Delta_\iota}$ is $T$-equivariant.  Just as in the proof of \cite[Th.~1.6]{mcgertyKirwanSurjectivity2018}, this exhibits $\mathcal O_{\Delta_\iota}$, in $D^b\mathrm{Coh}^T(\fM\times\overline{\fM})$, as an iterated cone on bundles $A\boxtimes B$, where $A$ is a tensor product of bundles $S_\lambda(E_j)$ and $B$ is a tensor product of bundles $S_\mu(F_j)$, for tautological or trivial bundles $E_j$ on $\fM$ and $F_j$ on $\overline{\fM}$.  Since $(A\otimes\chi)\boxtimes B=A\boxtimes(B\otimes\chi)$ for a character $\chi$ of $T$, we can take $T$ to act on each $A$ with the standard action, and put the twists by characters into $B$.

  Now let $\mathcal F$ be a $T$-equivariant coherent sheaf on $\overline{\fM}$, and write $p_1,p_2$ for the projections of $\fM\times\overline{\fM}$.  Since $\iota$ is an open immersion, $\iota^*\mathcal F\cong Rp_{1*}(\mathcal O_{\Delta_\iota}\otimes^L p_2^*\mathcal F)$, and $p_1$ is proper because $\overline{\fM}$ is projective.  Thus we can express $\iota^*\mathcal F$ in $D^b\mathrm{Coh}^T(\fM)$ as an iterated cone on objects of the form $A\otimes U$, where $U$ is a complex of $T$-representations whose cohomology groups are the finite-dimensional vector spaces $H^k(\overline{\fM},\mathcal F\otimes B)$.  That is, $[\iota^*\mathcal F]$ is a linear combination of the classes $[A]$, with the standard $T$-action, with coefficients in $\KC_T(\mathrm{pt})=R(T)$.  This shows that the image of $\iota^*\colon\KC_T(\overline{\fM})\to\KC_T(\fM)$ lies in the image of the Kirwan map.  Since $\iota^*$ is surjective by \cite[\S5.2.14]{CG97}, the Kirwan map is surjective.

The last statement holds because $R(G_\Bv)$ is generated as a ring by the exterior powers of the defining representations $V_i$ and the inverses of their determinants.
\end{proof}
Following \cite[Section~2.1 and Definition~3.1]{KrylovPaegelowShlykov26}, but working over $\Z$, set
\[
   \SA =\Z[q^{\pm1},Q_1^{\pm1},\ldots,Q_r^{\pm1}],
 \qquad\text{and}\qquad
 \SAc=\SA[(q-1)^{-1}],
\]
and consider the Hecke algebra $\HSAcn=\Hn(\SAc)$ over $\SAc$
with parameters~$q$ and~$\bQ$.  We write $\SA_\C=\C\otimes_\Z\SA$ and
$\SA^\circ_\C=\C\otimes_\Z\SAc$ for the complex forms, which are the rings
used in \cite{KrylovPaegelowShlykov26}.

Let $\mathcal{G}(n,r)$ be the Gieseker variety, which is the Nakajima quiver
variety of the Jordan quiver (a single vertex with one loop) with dimension
vector~$n$ and framing~$r$.  By \cite{NY03}, $\mathcal{G}(n,r)$ is the moduli
space of framed torsion-free sheaves on $\mathbb{P}^2(\C)$ of rank $r$ and
second Chern class $n$.  Let $\Lrn$ denote the preimage of $0$ under the
affinization morphism $\mathcal{G}(n,r)\rightarrow \mathcal{G}_0(n,r)$.  The
torus $\TC=\mathbb{T}\times\mathbb{T}_r$, where $\mathbb{T}$ is the maximal
torus of $SL_2$ and $\mathbb{T}_r$ that of $GL_r$, acts on both.  We write
$\KT(X)$ for the Grothendieck group of $\TC$-equivariant coherent
sheaves on~$X$, so that $\KT(\mathrm{pt})=\SA$, and
$\KT(X;\C)=\C\otimes_\Z\KT(X)$.

Krylov, Paegelow and Shlykov \cite[Theorem~3.16]{KrylovPaegelowShlykov26}
proved that the equivariant $K$-theory of the Gieseker variety is isomorphic
to the Jucys--Murphy center of the cyclotomic Hecke algebra $\Hn(\SA_\C)$ as an
algebra over the equivariant $K$-theory of a point: there is an isomorphism
$\KT(\mathcal{G}(n,r);\C)\cong Z(\Hn(\SA_\C))^\JM$ sending the class of the $i$th
exterior power of the tautological bundle to $e_i(L_1,\dots,L_n)$.  There
exists a perfect pairing $\langle-,-\rangle$ between
$\KT(\mathcal{G}(n,r);\C)$ and $\KT(\Lrn;\C)$
\cite[(4.1)]{KrylovPaegelowShlykov26}, and by
\cite[Lemma~4.5]{KrylovPaegelowShlykov26} there is a unique isomorphism
\[
  \psi\colon\Frac(\SA_\C)\otimes_\SA\KT(\Lrn)\longrightarrow\Frac(\SA_\C)\otimes_\SA\Tr(\Hn(\SA))
\]
of $Z(\Hn(\SA_\C))$-modules which intertwines this pairing with the trace
pairing $(z,[h])\mapsto\tau(zh)$, where $\tau$ is the symmetrizing form of
$\Hn$ \cite[Lemma~2.6]{huCocenterCyclotomic2022}.  Krylov, Paegelow and
Shlykov formulated the following conjecture, which asserts that $\psi$
restricts to the lattices.  As throughout
\cite[Section~4]{KrylovPaegelowShlykov26}, $q\neq 1$. By \cref{rem:q-1}
below, the conjecture cannot hold at $q=1$.

\begin{conjecture}[Compare {\cite[Conjecture~4.6]{KrylovPaegelowShlykov26}}]
\label{conj:KPS}
There is a natural $Z(\HSAcn)$-module isomorphism
$\SAc\otimes_{\SA}\KT(\Lrn)\cong\Tr(\HSAcn)$.
\end{conjecture}

The Gieseker variety is the quiver variety of the Jordan quiver with $\Bv=n$ and $\Bw=r$, so $G_\Bv=GL_n$ and $G_\Bw=GL_r$.  Here $E$ is the single loop, and $(\Gm)^E$ is the maximal torus $\mathbb T$ of $SL_2$, scaling $B_1$ and $B_2$ by inverse scalars.  So $\TC=\mathbb T\times\mathbb T_r$ acts through its inclusion in $(\Gm)^E\times G_\Bw$, with trivial component in $\mathbb{S}$, and \cref{lem:Kirwan-K} shows that the Kirwan map $R(GL_n\times\TC)\to\KT(\mathcal{G}(n,r))$, which sends a representation to the associated vector bundle, is surjective.  After extending scalars to $\C$, this is \cite[Proposition~3.13]{KrylovPaegelowShlykov26}, which Krylov, Paegelow and Shlykov deduce from \cite[Cor.~1.4]{mcgertyKirwanSurjectivity2018} by localization. The following proposition proves a generalization to the integral case of \cite[Conjecture~4.6]{KrylovPaegelowShlykov26}. Note that we add the part $\SAc\otimes_{\SA}$ before $\KT(\Lrn)$, which is necessary.

\begin{proposition}\label{P:MainC}\hfill
  \begin{enumerate}
    \item Let $R$ be a commutative domain in which $q$, $Q_1,\dots,Q_r$ and
      $q-1$ are invertible.  Then $Z(\Hn(R))=Z(\Hn(R))^\JM$, and this is a
      free $R$-module of rank $|\Parts|$.
    \item The isomorphisms of
      \cite[Theorem~3.16 and Lemma~4.5]{KrylovPaegelowShlykov26} restrict to
      isomorphisms of $Z(\HSAcn)$-modules
      \[
        \SAc\otimes_{\SA}\KT(\mathcal{G}(n,r))\cong Z(\HSAcn)
        \qquad\text{and}\qquad
        \SAc\otimes_{\SA}\KT(\Lrn)\cong\Tr(\HSAcn).
      \]
      In particular, \cref{conj:KPS} holds.
   \end{enumerate}
\end{proposition}

\begin{proof}
  Let $F$ be an $R$-algebra that is also a field. Then
  $\Hn(F)\cong F\otimes_{R}\Hn(R)$. Since $q-1$ is invertible in
  $R$, we have $q\neq1$ in $F$.
  Therefore, $Z(\Hn(F))=Z(\Hn(F))^\JM$ by \cref{thm:B}.
  The following diagram commutes:
  \[
     \begin{CD}
        F\otimes_{R}Z(\Hn(R))^\JM
         @>{\id_F\otimes\iota}>>
        F\otimes_{R}Z(\Hn(R))
         @>{\id_F\otimes \pi}>>
        F\otimes_{R}
         \dfrac{Z(\Hn(R))}{Z(\Hn(R))^\JM}
         @>>>0\\
         @V{p_0}VV @V{\wr}VV @. @.\\
        Z\bigl(\Hn(F)\bigr)^\JM @= Z\bigl(\Hn(F)\bigr)
         @. @.
    \end{CD}
  \]
  where $\iota\colon Z(\Hn(R))^\JM\hookrightarrow Z(\Hn(R))$ is
  inclusion, $\pi\colon Z(\Hn(R))\to Z(\Hn(R))/Z(\Hn(R))^\JM$ the natural quotient map, and the vertical maps are the natural base-change maps; the second is an isomorphism by \cite[Theorem~1.4(2)]{huCocenterCyclotomic2022}.
  By definition, $p_0$ is surjective.
  Hence, $\id_F\otimes\iota$ is surjective, so
  \[
    F\otimes_{R}\bigl(Z(\Hn(R))/Z(\Hn(R))^\JM\bigr)=0.
  \]
  By \cite[Theorem~1.4(2)]{huCocenterCyclotomic2022}, $Z(\Hn(R))$ is a
  free $R$-module of rank $|\Parts|$, so in particular it is finitely generated.  Taking $F=R/\mathfrak{m}$ for each
  maximal ideal $\mathfrak{m}$ of $R$, Nakayama's lemma shows that
  $Z(\Hn(R))/Z(\Hn(R))^\JM=0$.  Therefore,
  $Z(\Hn(R))^\JM=Z(\Hn(R))$, which is free of rank $|\Parts|$.  This proves (1).

  For (2), let $R=\SAc$.  By
  \cite[Lemma~7.1.5]{nakajimaQuiverVarieties2001a} (see also
  \cite[Lemma~3.10]{KrylovPaegelowShlykov26}), $\KT(\mathcal{G}(n,r))$ and
  $\KT(\Lrn)$ are free $\SA$-modules, so they are lattices in
  $\KT(\mathcal{G}(n,r);\C)$ and $\KT(\Lrn;\C)$.  The ring $R(GL_n\times\TC)$ is
  generated over $\SA$ by the exterior powers $\Lambda^i\C^n$, for
  $1\leq i\leq n$, and by $(\Lambda^n\C^n)^*$.  By \cref{lem:Kirwan-K} and
  \cite[Theorem~3.16]{KrylovPaegelowShlykov26}, the isomorphism
  $\KT(\mathcal{G}(n,r);\C)\cong Z(\Hn(\SA_\C))^\JM$ therefore maps
  $R\otimes_{\SA}\KT(\mathcal{G}(n,r))$ onto the $R$-subalgebra of
  $Z(\Hn(\SA^\circ_\C))$ generated by $e_1(L),\dots,e_n(L)$ and
  $e_n(L)^{-1}=(L_1\cdots L_n)^{-1}$.  These elements lie in $Z(\Hn(R))$,
  so this subalgebra lies between $Z(\Hn(R))^\JM$ and $Z(\Hn(R))$, which are
  equal by (1).  This gives the first isomorphism.

  The proof of \cite[Lemma~7.1.5]{nakajimaQuiverVarieties2001a} shows that
  $\langle-,-\rangle$ restricts to a perfect pairing between
  $\KT(\mathcal{G}(n,r))$ and $\KT(\Lrn)$ with values in $\SA$:
  with respect to the Bia{\l}ynicki-Birula decomposition, it is block
  triangular, with perfect pairings on the diagonal.  Since $\tau$ is a
  symmetrizing form and $\Tr(\Hn(R))$ is a free $R$-module of finite rank
  (\cite[Th.~1.4(1)]{huCocenterCyclotomic2022}), the trace pairing is
  perfect over $R$.  Dualizing the first isomorphism therefore gives an
  isomorphism $R\otimes_{\SA}\KT(\Lrn)\cong\Tr(\Hn(R))$ which
  intertwines the two pairings.  After extending scalars to $\Frac(\SA_\C)$ it
  agrees with $\psi$, since both intertwine the pairings and the trace
  pairing is nondegenerate.  Hence $\psi$ restricts to an isomorphism
  $R\otimes_{\SA}\KT(\Lrn)\cong\Tr(\Hn(R))$, which is
  $Z(\Hn(R))$-linear because $\psi$ is.  This verifies
  \cite[Conjecture~1.2; Conjecture~4.6]{KrylovPaegelowShlykov26}.  
\end{proof}

\begin{remark}\label{rem:q-1}
  The hypothesis that $q-1$ is a unit cannot be dropped: by the base change statement of \cite[Th.~1.4(2)]{huCocenterCyclotomic2022}, the equality $Z(\Hn(R))=Z(\Hn(R))^\JM$ for $R=\SA$ would specialize to every field with $q=1$.  This fails already for $r=1$, $n\geq 2$ and $q=Q_1=1$: then $\Hn(\kk)=\kk\Sym_n$ and $L_1=\dots=L_n=1$, so the Jucys--Murphy center is $\kk$, whereas $Z(\kk\Sym_n)$ has a basis of class sums.  See also Ariki's example \cite[\S 2]{Ariki:can}.
\end{remark}

Let $X$ be a smooth quasi-projective complex variety.  Write
$K_0^{\mathrm{vb}}(X)$ for the Grothendieck group of algebraic vector bundles and
$G_0(X)$ for the Grothendieck group of coherent sheaves.  The natural Cartan
homomorphism is
\[
 K_0^{\mathrm{vb}}(X)\longrightarrow G_0(X),\qquad [E]\longmapsto[E].
\]
By \cite[Section~15.1, pp.~281--282]{fultonIntersectionTheory1998}, this is an isomorphism
because $X$ is smooth. Hence, there is an isomorphism between the
$K$-theory and Grothendieck group of $X$:
\begin{equation}\label{eq:ordinary-K-definition}
 K(X)=K_0^{\mathrm{vb}}(X)
 \cong G_0(X).
\end{equation}
The product in $K(X)$ is induced by taking tensor products of vector
bundles.  For a commutative ring $\kk$, we write $K(X;\kk)=\kk\otimes_\Z K(X)$.

Now let $\kk$ be a field of characteristic not dividing $e$, let $\zeta\in\kk$ be a primitive $e$th root of unity, let $\charge=(\kappa_1,\dots,\kappa_r)\in\Z^r$ be a multicharge, and let $\HL[n](\kk)=\Hn(\kk)$ be the cyclotomic Hecke algebra with $q=\zeta$ and $Q_m=\zeta^{\kappa_m}$.  Then $I=\Z/e\Z$, with $c\in\Z/e\Z$ corresponding to $\zeta^c$, $\Lambda=\sum_m\Lambda_{\bar\kappa_m}$, where $\bar\kappa_m$ is the residue of $\kappa_m$ modulo $e$, and for $\alpha\in Q^+_n$ we write $\HL[\alpha](\kk)$ for the summand of $\HL[n](\kk)$ of \cref{C:blocks}, which is identified with $\Rn[\alpha](\kk)$ by \cref{thm:BK}, and $1_\alpha$ for its identity element.  Write $\Lambda=\sum_iw_i\Lambda_i$ and $\alpha=\sum_iv_i\alpha_i$.
In the notation of \cite[\S 5.3]{KrylovPaegelowShlykov26} and \cite{Paegelow26}, this is $\ell=e$ and $s_m=\bar\kappa_m$, the group $\Gamma_{\mathbf s}$ is generated by the element $g$ in the proof below, and $\mathbf d=\alpha$: the component $\M^{\Gamma_{\mathbf s}}_{\mathbf d}$ of the fixed locus is $\NLa$, and the block $B_{\mathbf d}$ is $\HL[\alpha](\C)$ accounting for different conventions in the Hecke relations.

\begin{proof}[Proof of \cref{thm:C}]
Parts (1) and (2) follow from Proposition 5.3.
For (3), first let $\kk=\C$.  In this case, (3) is \cite[Theorem~5.4]{KrylovPaegelowShlykov26}, whose hypothesis \cite[Conjecture~4.6]{KrylovPaegelowShlykov26} holds by \cref{P:MainC}, since their proof uses it only through \cite[Proposition~4.7]{KrylovPaegelowShlykov26}.  Since the isomorphism there is left a bit implicit, let us explain it in more detail, following \cite[\S 5.3]{KrylovPaegelowShlykov26}.

We can consider the tuple $g=(q,Q_1,\dots,Q_r)=(\zeta, \zeta^{\kappa_1}, \dots, \zeta^{\kappa_r})$ as an element of $\TC$.
The fixed locus $\mathcal{G}(n,r)^g$ is the disjoint union of its irreducible components, which are the quiver varieties of the cyclic quiver with framing $\Bw$ and highest weight $\Lambda$ \cite[Theorem~2.15 \& Proposition~2.17]{Paegelow26}.  The dimension vectors that occur are those $\alpha\in Q^+_n$ with $\NLa\neq\emptyset$, so $K(\mathcal{G}(n,r)^g;\C)=\bigoplus_\alpha K(\NLa;\C)$ \cite[Proposition~5.3]{KrylovPaegelowShlykov26}.  Let $\C_g$ denote $\C$ regarded as an $\SA$-algebra via $q\mapsto\zeta$ and $Q_m\mapsto\zeta^{\kappa_m}$; since $\zeta\neq1$, this factors through $\SAc$.  Since $g$ has order $e$, it acts fiberwise on the restriction of a $\TC$-equivariant vector bundle $E$ to $\mathcal{G}(n,r)^g$, which splits into eigenbundles $E|_{\mathcal{G}(n,r)^g}=\bigoplus_{c\in\Z/e\Z}E_c$, with $g$ acting on $E_c$ by $\zeta^c$.  The specialization map
\begin{equation}\label{eq:sp-g}
  \mathrm{sp}_g\colon\C_g\otimes_{\SA}\KT(\mathcal{G}(n,r))\longrightarrow K(\mathcal{G}(n,r)^g;\C),\qquad [E]\longmapsto\sum_{c\in\Z/e\Z}\zeta^c[E_c]
\end{equation}
is a ring homomorphism; by the localization theorem it is an isomorphism, but we will not need this.  Composing it with \cref{P:MainC}(2) and the base change isomorphism of \cite[Theorem~1.4(2)]{huCocenterCyclotomic2022}, we obtain a ring homomorphism
\[
  \Phi\colon Z(\HL[n](\C))\cong\C_g\otimes_{\SAc}Z(\HSAcn)\cong\C_g\otimes_{\SA}\KT(\mathcal{G}(n,r))\xrightarrow{\ \mathrm{sp}_g\ }\bigoplus_{\alpha\in Q^+_n}K(\NLa;\C)
\]
with $\Phi(e_i(L_1,\dots,L_n))=\mathrm{sp}_g[\Lambda^i\mathcal{V}]$, where $\mathcal{V}$ is the tautological bundle of $\mathcal{G}(n,r)$.

On $\NLa$ the tautological bundle splits as $\mathcal{V}=\bigoplus_cV_c$, where the $V_c$ are the tautological bundles of the quiver variety, so $\mathrm{sp}_g[\Lambda^i\mathcal{V}]$ is the $i$th elementary symmetric function in the elements $\zeta^cx$, where $x$ runs over the $K$-theoretic Chern roots of $V_c$.  On the Hecke side, $L_ae(\mathbf{i})=\zeta^{i_a}(1+y_a)e(\mathbf{i})$ by \cref{thm:BK} with $b(u)=1+u$.
As shown in \cref{eq:JM-colored-nilpotent-action,eq:JM-colored-symmetric-action}, the idempotent $1_\alpha$ and the elements $z_{c;j}=\sum_{\mathbf{i}\in I^\alpha}e_j(L_a\mid i_a=c)e(\mathbf{i})$ are explicit symmetric polynomials in $L_1,\dots,L_n$.  Since $x-1$ and the $y_a$ are nilpotent, the same polynomials in the roots $\zeta^cx$ give the idempotent of $\NLa$ and the classes $\zeta^{cj}[\Lambda^jV_c]$.  Hence $\Phi$ maps $Z(\HL[\alpha](\C))$ to $K(\NLa;\C)$ (this recovers \cite[Proposition~5.3]{Paegelow26}), sending $\zeta^{-cj}z_{c;j}=\sum_{\mathbf{i}}e_j(1+y_a\mid i_a=c)e(\mathbf{i})$ to $[\Lambda^jV_c]$; that is, $e_{c;j}$ goes to $e_j(x-1)$, over the roots $x$ of $V_c$.  The classes $[\Lambda^jV_c]$, together with the inverses of the line bundles $\Lambda^{v_c}V_c$, generate $K(\NLa)$ as a ring by \cref{lem:Kirwan-K}, and the image of $\Phi$, being a finite-dimensional algebra, contains the inverse of each of its units.  So $\Phi\colon Z(\HL[\alpha](\C))\to K(\NLa;\C)$ is surjective.  Both sides have dimension $|\Parts[\alpha]|$.  Hence $\Phi$ restricts to an isomorphism $Z(\HL[\alpha](\C))\cong K(\NLa;\C)$.  This proves (3) for $\kk=\C$.

Now let $\kk$ be arbitrary.   Let $\zeta_{\C}=\exp(2\pi i/e)$ and let $\mathscr{O}=\Z[\zeta_{\C},1/e]\subset\C$, the ring of integers of $\Q(\zeta_\C)$ with $e$ inverted.  Since the characteristic of $\kk$ does not divide $e$, there is a ring homomorphism $\mathscr{O}\to\kk$ with $\zeta_\C\mapsto\zeta$.

The element $\zeta_\C-1$ is a unit in $\mathscr{O}$, since $\prod_{k=1}^{e-1}(1-\zeta_\C^k)=e$.  So \cref{P:MainC}(1) applies to $R=\mathscr{O}$ with $q=\zeta_\C$ and $Q_m=\zeta_\C^{\kappa_m}$: $Z(\HL[n](\mathscr{O}))$ is a free $\mathscr{O}$-module, and by \cite[Theorem~1.4(2)]{huCocenterCyclotomic2022} its base changes to $\C$ and to $\kk$ are $Z(\HL[n](\C))$ and $Z(\HL[n](\kk))$.

Furthermore, the idempotent $1_\alpha$ and the elements $z_{c;j}$ are defined over $\mathscr{O}$: the generalized-eigenspace projectors are polynomials in the $L_a$ with coefficients in $\mathscr{O}$, because the differences $\zeta_\C^c-\zeta_\C^{c'}$ for $c\neq c'$ in $\Z/e\Z$ are units in $\mathscr{O}$.  So $1_\alpha Z(\HL[n](\mathscr{O}))$ is a lattice in $Z(\HL[\alpha](\C))$ whose base change to $\kk$ is $Z(\HL[\alpha](\kk))$.

The isomorphism $\Phi$ maps $1_\alpha Z(\HL[n](\mathscr{O}))$ into $K(\NLa;\mathscr{O})$: by \cref{P:MainC}(1), $Z(\HL[n](\mathscr{O}))$ is generated by $e_1(L),\dots,e_n(L)$, whose images $\mathrm{sp}_g[\Lambda^i\mathcal{V}]$ have coefficients in $\mathscr{O}$.  The image contains the classes $[\Lambda^jV_c]=\Phi(\zeta_\C^{-cj}z_{c;j})$, and also the inverse of $[\Lambda^{v_c}V_c]$, since $z_{c;v_c}$ is invertible in $1_\alpha Z(\HL[n](\mathscr{O}))$, its inverse being the same sum with each $L_a$ replaced by $L_a^{-1}$.  By \cref{lem:Kirwan-K}, the map to $K(\NLa;\mathscr{O})$ is surjective; it is injective since it is the restriction of an isomorphism over $\C$.  Tensoring the resulting isomorphism $1_\alpha Z(\HL[n](\mathscr{O}))\cong K(\NLa;\mathscr{O})$ with $\kk$ over $\mathscr{O}$ gives $Z(\HL[\alpha](\kk))\cong K(\NLa;\kk)$.
\end{proof}

\subsection*{Acknowledgements}

We thank Vasily Krylov, Rapha\"el Paegelow and Pavel Shlykov for helpful discussions on their work \cite{KrylovPaegelowShlykov26}.
JH was supported by the National Natural Science Foundation of
China (No. 12431002). AM was partially supported by the
Australian Research Council DP240101809. BW was supported by the NSERC through a Discovery Grant. This research was supported in part by Perimeter Institute for Theoretical Physics. Research at Perimeter Institute is supported in part by the Government of Canada through the Department of Innovation, Science and Economic Development Canada and by the Province of Ontario through the Ministry of Colleges and Universities. We acknowledge the use of AI in
proofreading this manuscript.

\printbibliography

\end{document}

%% file: preamble.tex
\usepackage[english]{babel}
\usepackage{csquotes}
\usepackage[letterpaper,top=2cm,bottom=2cm,left=3cm,right=3cm,marginparwidth=1.75cm]{geometry}
\usepackage{tikz}
\usepackage{quiver}
\usepackage{xparse}
\usepackage{enumitem}
\usepackage[maxbibnames=99,backend=biber,url=false,style=alphabetic,
  sorting=nyt,doi=true,eprint=false]{biblatex}
\renewbibmacro*{url+urldate}{%
}
\AtEveryBibitem{
  \clearfield{isbn}
  \clearfield{issn}
  \clearlist{location}
  \clearfield{month}
  \clearfield{day}
  \clearfield{series}
  \clearfield{pages}
  \clearlist{language}
  \clearlist{translator}
  \clearfield{urldate}
  \clearfield{pubstate}
  \clearfield{url}
  \renewbibmacro{in:}{}
}
\usepackage{amsmath,amsfonts,amssymb,xcolor}

\usepackage{graphicx}
\usepackage[colorlinks=true,allcolors=blue]{hyperref}
\usepackage{zref-clever}
\usepackage{aliascnt}
\zcsetup{cap}
\providecommand{\cref}[1]{\zcref{#1}}
\providecommand{\Cref}[1]{\zcref{#1}}

\zcRefTypeSetup{equation}{
  name-sg=,
  Name-sg=,
  name-pl=,
  Name-pl=
}

\NewDocumentCommand{\NewTheoremWithZref}{ m o m m }{%
  \IfNoValueTF{#2}{%
    \newtheorem{#1}{#3}%
    \zcRefTypeSetup{#1}{%
      name-sg=#3, Name-sg=#3, name-pl=#4, Name-pl=#4%
    }%
  }{%
    \newaliascnt{#1}{#2}%
    \newtheorem{#1}[#1]{#3}%
    \aliascntresetthe{#1}%
    \zcRefTypeSetup{#1}{%
      name-sg=#3, Name-sg=#3, name-pl=#4, Name-pl=#4%
    }%
  }%
}

\numberwithin{equation}{section}

\newtheorem{itheorem}{Theorem}
\zcRefTypeSetup{itheorem}{
  name-sg=theorem,
  Name-sg=Theorem,
  name-pl=theorems,
  Name-pl=Theorems
}
\newtheorem{theorem}{Theorem}[section]

\NewTheoremWithZref{corollary}[theorem]{Corollary}{Corollaries}
\NewTheoremWithZref{lemma}[theorem]{Lemma}{Lemmas}
\NewTheoremWithZref{proposition}[theorem]{Proposition}{Propositions}
\NewTheoremWithZref{assumption}[theorem]{Assumption}{Assumptions}
\NewTheoremWithZref{definition}[theorem]{Definition}{Definitions}
\NewTheoremWithZref{ques}[theorem]{Question}{Questions}
\NewTheoremWithZref{conjecture}[theorem]{Conjecture}{Conjectures}
\NewTheoremWithZref{algorithm}[theorem]{Algorithm}{Algorithms}

\theoremstyle{remark}
\NewTheoremWithZref{remark}[theorem]{Remark}{Remarks}
\NewTheoremWithZref{examplex}{Example}{Examples}

\newcommand{\Z}{\mathbb{Z}}
\newcommand{\Q}{\mathbb{Q}}
\newcommand{\C}{\mathbb{C}}
\newcommand{\F}{\mathbb{F}}
\newcommand{\Bv}{\mathbf{v}}
\newcommand{\Bw}{\mathbf{w}}

\newcommand{\Be}{\mathbf{e}}
\newcommand{\kk}{\Bbbk}

\newcommand{\al}{\alpha}
\newcommand{\fM}{\mathfrak{M}}
\newcommand{\fMo}{\fM(\Bv-n\Be_k,\Bw)}
\newcommand{\fMt}{\fM(\Bv,\Bw)}
\newcommand{\fMf}[2]{\fM_{#1;#2}(\Bv,\Bw)}
\newcommand{\fL}{\mathfrak{L}}
\newcommand{\fP}{\mathfrak{P}}
\newcommand{\tU}{\mathcal{U}}
\newcommand{\Tr}{\operatorname{Tr}}
\newcommand{\Sym}{\mathfrak{S}}
\newcommand{\SymR}{\operatorname{Sym}}

\newcommand{\cl}{\operatorname{cl}}
\newcommand{\NH}{\operatorname{NH}}

\renewcommand\Rn[1][n]{\mathsf{R}^\Lambda_{#1}}

\usepackage{mathtools}
\usepackage{mathrsfs}
\usepackage{amscd}

\DeclarePairedDelimiterX{\set}[1]{\{}{\}}{\setargs{#1}}
\NewDocumentCommand{\setargs}{>{\SplitArgument{1}{|}}m}{\setargsaux#1}
\NewDocumentCommand{\setargsaux}{mm}{\,\IfNoValueTF{#2}{#1}{#1\,\delimsize|\,\mathopen{}#2}\,}

\DeclareMathOperator\id{id}

\newcommand\blam{{\boldsymbol{\lambda}}}
\newcommand\bmu{{\boldsymbol{\mu}}}

\newcommand\JM{{\text{JM}}}
\newcommand\bQ{\mathbf{Q}}

\newcommand\NLa{\M_{\Lambda,\alpha}}

\newcommand{\M}{\mathfrak{M}}

\newcommand{\HH}{\mathscr{H}}
\newcommand\Hn[1][n]{\HH^{\bQ}_{q,#1}}

\newcommand\HL[1][n]{\HH^\Lambda_{\zeta,#1}}

\newcommand{\KC}{\mathcal{K}}
\newcommand{\LC}{\mathcal{L}}
\newcommand{\TC}{\mathcal{T}}
\newcommand\KT{\KC_{\TC}}
\newcommand\Lrn{\LC(n,r)}

\newcommand\SA{{\mathscr{A}}}
\newcommand\SAc{{\SA^\circ}}

\newcommand\HSAcn{\HH^\circ_{q,n}}

\newcommand\PP{\mathscr{P}}
\NewDocumentCommand\Parts{ D(){r} O{n}}{\PP^{#1}_{#2}}
\NewDocumentCommand\bParts{D(){s,r}O{b,n}}{\PP^{#1}_{#2}}
\newcommand\charge{{\boldsymbol{\kappa}}}
\DeclareMathOperator{\Frac}{Frac}
\DeclareMathOperator{\Hom}{Hom}
\newcommand{\Gm}{\mathbb{G}_m}